\documentclass[10pt,a4paper]{article}
\usepackage[numbers,sort&compress]{natbib}
\usepackage[T1]{fontenc}
\usepackage[utf8]{inputenc}
\usepackage{authblk}
\usepackage{amsmath,amssymb,amsthm,esint,bm}
\usepackage{mathrsfs}
\usepackage{bookmark}
\usepackage{amsmath}
\usepackage{amsmath, amsthm}
\usepackage{enumitem}
\allowdisplaybreaks[4]

\newtheorem{definition}{Definition}[section]
\newtheorem{theorem}[definition]{Theorem}
\newtheorem{lemma}[definition]{Lemma}
\newtheorem{proposition}[definition]{Proposition}

\theoremstyle{remark}
\newtheorem{remark}[definition]{Remark}
\numberwithin{equation}{section}

\newcommand{\abs}[1]{\lvert#1\rvert}
\newcommand{\Abs}[1]{\left\lvert#1\right\rvert}

\newcommand{\rn}{{\mathbb{R}^d}}

\allowdisplaybreaks[4]

\title{Borderline gradient continuity for degenerate/singular fully nonlinear elliptic equations with Hamiltonian terms}

\author[a]{Wentao Huo}
\author[b]{Lingwei Ma}
\author[b]{Zhenqiu Zhang\thanks{Corresponding author.}}
\affil[a]{School of Mathematical Sciences, Nankai University, Tianjin 300071, P.R. China}
\affil[b]{School of Mathematical Sciences and LPMC, Nankai University, Tianjin 300071, P.R. China}
\date{\today}

\usepackage{hyperref}
\begin{document}
\maketitle
\footnotetext[1]{E-mail: huowentaoouc@163.com (W. Huo), malingwei@nankai.edu.cn (L. Ma), zqzhang@nankai.edu.cn (Z. Zhang).}

\begin{abstract}
This paper focuses on a class of fully nonlinear elliptic equations with general double phase degeneracy/singularity law and Hamiltonian terms of the form
\begin{equation*}
	\Phi(|Du|,x)F(D^2 u, x)+H(Du,x) =f(x)
	\quad  \text{in} \quad B_{1},
\end{equation*}
where $\Phi$ takes one of two typical forms:
 $$\Phi(|Du|,x)=\sigma_{1}(|Du|)+a(x)\sigma_{2}(|Du|)\quad {\rm or}\quad \Phi(|Du|,x)=\frac{\sigma_{1}(|Du|)}{|Du|}+a(x)\frac{\sigma_{2}(|Du|)}{|Du|}.$$
Under suitable assumptions on the operator $F$, Hamiltonian term $H$, source term $f$ and modulating coefficient $a$,  we establish $C^{1}$ regularity for viscosity solutions, provided that $\sigma_{1},\sigma_{2}$ are moduli of continuity and their inverses are Dini continuous. Our argument is based on a tangential analysis via approximating hyperplanes combined with a new recursive renormalization algorithm adapted to the present framework. It is noteworthy that our results are new even for the case $a(x)\equiv 0$.

Mathematics Subject classification (2020):  35B65; 35J60; 35J70; 35J75; 35D40.

Keywords: Gradient regularity; degenerate/singular fully nonlinear equations; viscosity solution; Hamiltonian terms; Dini continuity.

\end{abstract}


\section{Introduction}\label{section1}
This paper is devoted to studying the regularity of viscosity solutions to the following degenerate or singular fully nonlinear elliptic equations with Hamiltonian terms
\begin{equation}\label{4doublephasezhufangcheng}
\Phi(|Du|,x)F(D^2 u, x)+H(Du,x) =f(x) \quad  \text{in} \quad B_{1},
\end{equation}
where $B_{1}:=B_{1}(0)$ is a unit ball in the Euclidean space $\mathbb{R}^d$, $F:S^{d}\times B_{1}\rightarrow \mathbb{R}$ is a second order fully nonlinear uniformly elliptic operator, the function $\Phi$ features singularity or degeneracy as the gradient vanishes, the Hamiltonian term $H$ fulfills some appropriate conditions, and source term $f\in C({B_{1}}) \cap L^{\infty}(B_{1})$.

A natural starting point is to consider $\Phi$ of the double phase form
\begin{equation}\label{phiform1}
	\Phi(|Du|,x):=\sigma_{1}(|Du|)+a(x)\sigma_{2}(|Du|)
\end{equation}
with degeneracy rates $\sigma_{1},\sigma_{2}$ and a modulating function $0\leq a(\cdot)\in C(B_{1})$.
Our primary goal is to examine minimal conditions on $\sigma_{1},\sigma_{2}$ to guarantee solutions retain $C^{1}$-differentiability. As is common in many branches of mathematical analysis, $C^{1}$ estimate is more challenging, since it represents a critical borderline regularity.

One of the motivations for studying this class of equations arises from certain variational integrals of the calculus of variations with nonstandard growth satisfying a sort of double-phase structure \cite{Marcellini1989,Mingione2015,Mingione201522,Zhikov1993,Zhikov1995}. In addition, the study of equations like \eqref{4doublephasezhufangcheng} arises naturally in the optimal stochastic control problems \cite{Fleming,Leoni,Birindelli2019,Lions1989} and as a generalization of Hamilton–Jacobi equations \cite{Armstrong2015,Pimentel2023}.

The theory of equations of the form \eqref{4doublephasezhufangcheng} has experienced an impressive progress through the past decades, starting with the pioneering works \cite{Birindelli2006,Birindelli2007CPAA,Birindelli2004}, including aspects such as existence and uniqueness, comparison principle, Harnack inequalities and H\"{o}lder regularity; further significant results are contained in \cite{Imbert2011JDE,Junges2010,Silvestre2016JEMS}. Afterwards, $C^{1,\alpha}$ regularity of viscosity solutions to degenerate or singular elliptic equations has been the subject of intensive study in regularity theory, with one of the most widely studied prototypes
given by 
\begin{equation*}
	\abs{Du}^{p}F(D^2 u)+H(Du,x)=f(x) \quad  \text{in} \quad B_{1}
\end{equation*}
with $p>-1$, see \cite{Imbert1,Ricarte,Birindelli2014ESAIM,Birindelli2015,B-Demengel2016,B-Demengel2019,Andrade,Bronzi2020}. It is noteworthy that
such regularity theory for viscosity solutions plays a crucial role the study of various free boundary problems of obstacle type \cite{SilvaDCDS,SilvaRM}, of one-phase Bernoulli type \cite{dasilva2023}, and of singular perturbation type \cite{Bezerra2023,Bezerra2025}, just to cite a few. 

A generalisation of the degeneracy law $\xi\rightarrow |\xi|^{p}$ appeared in \cite{Andrade2022}. In that paper, Andrade et al. considered viscosity solutions to degenerate fully nonlinear equations of the form
\begin{equation*}
	\sigma(\abs{Du})F(D^2 u)=f(x) \quad  \text{in} \quad B_{1}
\end{equation*}
under the assumptions that  $\sigma:[0,\infty)\rightarrow [0,\infty)$ is a general modulus of continuity whose inverse $\sigma^{-1}$ is Dini continuous, and $f\in L^{\infty}(B_1)$. In this setting, they proved that solutions are locally of class $C^{1}$.

In the context of fully nonlinear elliptic equations with nonhomogeneous degeneracy, the pioneering work of De Filippis \cite{Fili} considered equations of the form
\begin{equation*}
	\bigg(|Du|^{p}+a(x)|Du|^{q}\bigg) F(D^2 u) =f(x)
	\quad  \text{in} \quad B_{1}
\end{equation*}
for $0\leq p\leq q$, $0\leq a(\cdot)\in C(B_{1})$ and $f\in C(B_{1})\cap L^{\infty}(B_{1})$. Under such conditions, the local $C^{1,\alpha}$ regularity of viscosity solutions was established.  
Subsequently, da Silva and Ricarte \cite{Silva2020}
improved De Filippis’ result and addressed a variety of applications in nonlinear elliptic models and related free boundary problems. 
Furthermore, regarding the sharp gradient H\"{o}lder regularity for fully nonlinear elliptic equations with unbalanced variable degeneracy, see \cite{Fang,Bezerra2023}.

For very general classes of degeneracy/singularity, we refer to recent work \cite{Baasandorj,Huo2026} in the context of fully nonlinear elliptic equations. Specifically, under the key assumptions that the function $\Phi:[0,\infty)\times\Omega\rightarrow [0,\infty)$ is a continuous map and that there exist constants $s(\Phi)\geq i(\Phi)>-1$ such that the map $t\mapsto\frac{\Phi(t,x)}{t^{i(\Phi)}}$ is almost non-decreasing and the map $t\mapsto \frac{\Phi(t,x)}{t^{s(\Phi)}}$ is almost non-increasing in $(0,\infty)$, optimal local $C^{1,\alpha}$ regularity for viscosity solutions of \eqref{4doublephasezhufangcheng} has been obtained in \cite{Baasandorj} (without Hamiltonian term) and in \cite{Huo2026} (with Hamiltonian term).

Inspired by the previous literature, in this paper, we naturally investigate $C^{1}$
regularity of viscosity solutions to \eqref{4doublephasezhufangcheng} with $\Phi(|Du|,x)$ given by \eqref{phiform1}, provided that degeneracy laws $\sigma_{1},\sigma_{2}$ fulfill minimal assumption \eqref{A3} below. A by-product of our argument is an explicit modulus of continuity for $Du$. To the best of our knowledge, there are no results on the borderline gradient continuity of solutions to varying coefficient fully nonlinear elliptic equations simultaneously involving double-phase type degeneracy laws and Hamiltonian terms.

Before stating our main results, we start by presenting the structural assumptions on \eqref{4doublephasezhufangcheng} to be employed throughout this work:
\begin{enumerate}[label=(\text{\bf H}\arabic{enumi}),ref=\textbf{H}\arabic{enumi}]
	\item \label{A1} The fully nonlinear operator $F:S^{d}\times B_{1}\rightarrow \mathbb{R}$ is uniformly $(\lambda,\Lambda)$-elliptic, i.e., there exist constants $0<\lambda\leq \Lambda$ such that
	$$\lambda\|N\|\leq F(M+N,x)-F(M,x)\leq \Lambda\|N\|$$
	for any $M,N\in S^{d}$ with $N\geq 0$ and $x\in B_{1}$. For normalization purposes, we assume that $F(0,x) = 0$ for all $x\in B_{1}$.
	\item \label{A2}
	There exist constants $C>0$ and $\theta\in(0,1)$ such that
	\begin{equation*}
		{\rm osc}_{F}(x,y):=\sup\limits_{M\in S^{d}\setminus \{0\}}\frac{\Abs{F(M,x)-F(M,y)}}{\|M\|}\leq C|x-y|^{\theta}
	\end{equation*}
	for all $x,y\in B_{1}$. For convenience,
	we denote ${\rm osc}_{F}(x):={\rm osc}_{F}(x,0)$ and 
	$$C_{F}:=\inf\left\{C>0:{\rm osc}_{F}(x,y)\leq C|x-y|^{\theta},\;\forall\, x,y\in B_{1}\right\}.$$
	\item \label{A3} $\sigma_{i}:[0,\infty)\rightarrow [0,\infty)$, $i=1,2$, are moduli of continuity satisfying
	\begin{equation}\label{guanxinormalization}
		\sigma_{1}(t)\geq\sigma_{2}(t) \quad {\rm for\;all\;} t\in[0,1],
	\end{equation}
	and $\sigma_{2}$ admits an inverse $\sigma_{2}^{-1}$ that is Dini continuous.
	\item \label{A4} The modulating coefficient $a(\cdot)$ satisfies $0\leq a(\cdot)\in C(B_{1})$.
	\item \label{A5} The Hamiltonian term $H:\mathbb{R}^{d}\times B_{1}\rightarrow \mathbb{R}$ is continuous, and there is a constant $\mathcal{M}>0$ such that
	\begin{equation*}
		\Abs{H(\xi,x)}\leq \mathcal{M}\left(1+\sigma_{1}(|\xi|)\right)
	\end{equation*}
	for every $\xi\in \mathbb{R}^{d}$ and $x\in B_{1}$.
	\item \label{A6} The source term $f$ belongs to $C({B_{1}}) \cap L^{\infty}(B_{1})$.
\end{enumerate}

Recall that a function $\omega:\mathbb{R}_{0}^{+}:=(0,\infty)\rightarrow \mathbb{R}_{0}^{+}$ is called a modulus of continuity if $\omega$ is monotone increasing and $\lim\limits_{t\rightarrow 0}\omega(t)=0$. We say that a modulus of continuity $\omega$ is Dini continuous, if 
	$$\int_{0}^{1}\frac{\omega(t)}{t}\text{d}t<\infty.$$

Note that $\sigma_{1}^{-1}$ is automatically Dini continuous, due to the Dini continuity of $\sigma_{2}^{-1}$, \eqref{guanxinormalization} and the monotonicity of $\sigma_{1},\sigma_{2}$.
In addition, Dini condition can also be characterized in terms of the summability of $\omega$ evaluated along geometric sequences, i.e., $\omega$ satisfies the Dini condition if and only if
	\begin{equation*}
		\sum_{k=1}^{\infty} \omega(\vartheta^{k}) < \infty
	\end{equation*}
	for every $\vartheta\in (0,1) $. For more details about this equivalence, we refer the readers to \cite[Definition 1]{Andrade2022}. 

We now state the first main result concerning the borderline regularity of solutions
to equation \eqref{4doublephasezhufangcheng} in the degenerate case.
\begin{theorem}\label{main}
	 Let $u\in C(B_{1})$ be a viscosity solution of \eqref{4doublephasezhufangcheng} with \eqref{phiform1} under assumptions \eqref{A1}-\eqref{A6}. Then $ u \in C^{1}_{\rm loc}(B_{1})$. Moreover, there exists a modulus of continuity $ \omega : \mathbb{R}_{0}^{+} \rightarrow  \mathbb{R}_{0}^{+} $ depending only on $d,\lambda,\Lambda,C_{F},\theta,\sigma_{1}, \sigma_{2},\|u\|_{L^{\infty}(B_{1})},\mathcal{M}$ and $\|f\|_{L^{\infty}(B_{1})}$ such that
	 \begin{equation*}
	 	|Du(x)-Du(y)| \leq \omega (|x-y|)
	 \end{equation*}
	 for every $ x,y \in B_{1/4} $. 
\end{theorem}

\begin{remark}
	 This result is new even for the case $a\equiv 0$. Observe that, we do not assume that the moduli of continuity are necessarily H\"{o}lder continuous. In particular, if $ \sigma_{i}^{-1}(i=1,2)$ behave like H\"{o}lder continuous functions near the origin, then solutions are locally $ C^{1,\alpha} $ for some $ 0 < \alpha <1 $. Therefore, our finding covers and extends the results of \cite{Andrade,Fili,B-Demengel2016,Birindelli2014ESAIM,Imbert1,Silva2020,Ricarte} in a unified way. 
\end{remark}
\begin{remark}
	Our finding can be
	viewed as an extension of the work \cite[Theorem 1]{Andrade2022}, now within the framework
	of Hamiltonian terms and double phase degenerate law as well as variable coefficients. It is worth emphasizing that this
	generalization of regularity results here is nontrivial. Indeed, due to the simultaneous presence of double-phase degenerate law and Hamiltonian term, the analytical tools in \cite{Andrade2022} cannot be applied directly and our analysis becomes more complicated and delicate. More precisely, we need to construct a new shored-up sequence of moduli of continuity tailored to our setting, which is a key ingredient towards the proof of Theorem \ref{main} based on iteration and convergence analysis, see Section \ref{sec3.3}.
\end{remark}
\begin{remark}
	An extension of our result above also holds to multi-phase degenerate fully nonlinear equation modelled by
	\begin{equation*}
		\bigg( \sigma_{0}(|Du|) + \sum_{i=1}^{N} a_{i}(x)\sigma_{i}(|Du|)                   \bigg)F(D^{2}u,x)+H(Du,x) = f(x) \quad \text{in} \quad B_{1},
	\end{equation*}
	where $0 \leq a_{i}(\cdot) \in C(B_{1})$, $\sigma_{0},\sigma_{i}:[0,\infty)\rightarrow [0,\infty)$ are moduli of continuity, $i\in\left\{1,2,\cdots,N\right\}$, 
	\begin{equation*} 
		 \sigma_{0}(t) \geq \sigma_{1}(t) \geq \cdots \geq \sigma_{N}(t) \;\; {\rm for\;all\;} t\in [0,1], 
	\end{equation*}
	and $\sigma_{N}^{-1}$ is Dini continuous.
\end{remark}
	
	However, there are many important examples of function $\Phi$ which are not covered by \eqref{phiform1} with $\sigma_{1},\sigma_{2}$ satisfying \eqref{A3}. For instance, given any $p>0$ and nonnegative continuous function $a(\cdot)$,
	\begin{equation}\label{example1}
		\Phi(t,x)=\frac{\left(e^{t}-1\right)^{p}}{t}+a(x)t^{p-1},\quad \Phi(t,x)=t^{p-1}+a(x)\frac{\log^{p}(1+t)}{t},
	\end{equation}
\begin{equation}\label{example2}
	\Phi(t,x)=\frac{\left(e^{t}-1\right)^{p}}{t}+a(x)\frac{\log^{p}(1+t)}{t},
\end{equation}
and many other examples. In addition, note that the examples above also do not fall in the realm of the paper \cite{Baasandorj,Huo2026}. Notably, Baasandorj et al. \cite{Baasandorj2023AML} investigated the $C^1$ regularity of viscosity solutions for degenerate/singular fully nonlinear elliptic equations of the form 
$$\Phi(|Du|)F(D^2 u) = f(x) \quad \text{in }\quad  B_1,$$
where $\Phi(|Du|):=\frac{\sigma(|Du|)}{|Du|}$, $\sigma:[0,\infty)\rightarrow [0,\infty)$ is a general modulus of continuity whose inverse $\sigma^{-1}$ is Dini continuous, and $\lim\limits_{t\rightarrow \infty}\sigma(t)=+\infty$. It is clear that $\Phi(t)=\frac{\left(e^{t}-1\right)^{p}}{t}$ and $\Phi(t)=\frac{\log^{p}(1+t)}{t}$ with $p>0$ fall into this framework. 
	
	At this point, a natural question is whether it is feasible to establish regularity results for equation \eqref{4doublephasezhufangcheng} with $\Phi$ of the form given in \eqref{example1} and \eqref{example2}.	
	To this end, we consider a class of $\Phi$ of the form
	\begin{equation}\label{xintiaojian}
		\Phi(|Du|,x)=\frac{\sigma_{1}(|Du|)+a(x)\sigma_{2}(|Du|)}{|Du|},
	\end{equation}
	where $\sigma_{1},\sigma_{2}$ and $a(\cdot)$ satisfy the same conditions given in \eqref{A3} and \eqref{A4}. In this setting, we examine the dichotomy: either
	\begin{equation}\label{4aaquyu0tiaojian}
	\lim\limits_{t\rightarrow 0^{+}}\frac{\sigma_{1}(t)}{t}=0
	\end{equation}
	or
	\begin{equation}\label{4aaquyuwuqiongtiaojian}
		\lim\limits_{t\rightarrow 0^{+}}\frac{\sigma_{2}(t)}{t}\geq \kappa_{0}>0.
	\end{equation}
Under the newly proposed conditions, we prove the differentiability of solutions to \eqref{4doublephasezhufangcheng} with $\Phi$ of the form \eqref{xintiaojian}.

Notably, since $\sigma_{1}(t)\geq \sigma_{2}(t)$ for all $t\in[0,1]$, \eqref{4aaquyu0tiaojian} implies that $	\lim\limits_{t\rightarrow 0^{+}}\frac{\sigma_{2}(t)}{t}=0$, and \eqref{4aaquyuwuqiongtiaojian} implies $\lim\limits_{t\rightarrow 0^{+}}\frac{\sigma_{1}(t)}{t}\geq \kappa_{0}>0$. Clearly, when $p>1$, the concrete examples given in \eqref{example1} and \eqref{example2} satisfy \eqref{xintiaojian}, \eqref{4aaquyu0tiaojian}, \eqref{A3} and \eqref{A4}; when $0<p\leq 1$, they fulfill \eqref{xintiaojian}, \eqref{4aaquyuwuqiongtiaojian}, \eqref{A3} and \eqref{A4}. Furthermore, observe that \eqref{4aaquyuwuqiongtiaojian} includes the singular case.

We are now in a position to state the second main result on $C^{1}$ regularity under assumptions \eqref{xintiaojian} and \eqref{4aaquyu0tiaojian}.
	\begin{theorem}\label{main22}
		Let $u\in C(B_{1})$ be a viscosity solution of \eqref{4doublephasezhufangcheng} with \eqref{xintiaojian}. Assume \eqref{A1}-\eqref{A6} and \eqref{4aaquyu0tiaojian} hold. Then $ u \in C^{1}_{\rm loc}(B_{1})$. Moreover, there exists a modulus of continuity $ \omega : \mathbb{R}_{0}^{+} \rightarrow  \mathbb{R}_{0}^{+} $ depending only on $d,\lambda,\Lambda,C_{F},\theta,\sigma_{1}, \sigma_{2},\|u\|_{L^{\infty}(B_{1})}$, $\mathcal{M}$ and $\|f\|_{L^{\infty}(B_{1})}$ such that
	\begin{equation*}
		|Du(x)-Du(y)| \leq \omega (|x-y|)
	\end{equation*}
	for every $ x,y \in B_{1/4} $. 
\end{theorem}
Finally, for singular or mildly degenerate case, we obtain local $C^{1,\alpha}$ regularity results as
follows.
\begin{theorem}\label{main333}
	Let $u\in C(B_{1})$ be a viscosity solution of \eqref{4doublephasezhufangcheng} with \eqref{xintiaojian}. Assume \eqref{A1}-\eqref{A6} and \eqref{4aaquyuwuqiongtiaojian} hold. Then $ u \in C^{1,\alpha}_{\rm loc}(B_{1})$ for some $\alpha\in(0,1)$ depending only upon $d$, $\lambda$ and $\Lambda$.  
\end{theorem}
\begin{remark}
Clearly, Theorems \ref{main22} and \ref{main333} encompass and generalize the main results of \cite{Baasandorj2023AML}(without Hamiltonian term), which considers constant
coefficient fully nonlinear elliptic equations with single degenerate/singularity rate. In addition, it is worth highlighting that our assumptions are weaker than those in \cite{Baasandorj2023AML}, since
we do not require $\lim\limits_{t\rightarrow \infty}{\sigma_{i}(t)}=\infty$, $i=1,2$. 
\end{remark}

The remainder of this paper is organized as follows. In Section \ref{section2}, we introduce the basic notions and some well-known results. Section \ref{section3} is devoted to proving the borderline gradient continuity stated in Theorem \ref{main}. In the last section, we complete the proof of Theorems \ref{main22} and \ref{main333} regarding $C^{1}$ and $C^{1,\alpha}$ regularity.
\section{Preliminaries}\label{section2}
Throughout this paper, let $S^{d}$ be the set of all real symmetric $d\times d$ matrices and $B_{r}(x_{0})$ be the open ball with radius $r$ and centred at $x_{0}\in\rn$. In particular, we shall simply denote $B_{r}:=B_{r}(0)$. Also, we have the usual partial ordering: $A\leq B$ in $S^{d}$ means that $\left\langle A\xi,\xi\right\rangle\leq \left\langle B\xi,\xi\right\rangle$ for any $\xi \in \mathbb{R}^{d}$. In other words, $B-A$ is positive semidefinite. The symbol $C$ denotes a positive constant whose value may vary from line to line, and only the relevant dependencies are specified in parentheses. Besides, a constant is said to be universal if it depends at most upon the structure constants appearing in assumptions \eqref{A1}-\eqref{A6}.

Let us begin with the definition of the Pucci extremal operators.
\begin{definition}
	Let $0<\lambda\leq \Lambda$ be as in \eqref{A1}. We define the extremal Pucci operators $P_{\lambda,\Lambda}^{\pm}:S^{d}\rightarrow \mathbb{R}$ as follows:
	$$P_{\lambda,\Lambda}^{+}(M):=\Lambda\sum\limits_{e_{i}>0}e_{i}+\lambda\sum\limits_{e_{i}<0}e_{i}$$
	and 
	$$P_{\lambda,\Lambda}^{-}(M):=\lambda\sum\limits_{e_{i}>0}e_{i}+\Lambda\sum\limits_{e_{i}<0}e_{i},$$
	where $\{e_{i}\}_{i=1}^{d}$ are the eigenvalues of the matrix $M$.
\end{definition}
With the Pucci operators in hand, the uniformly $(\lambda,\Lambda)$-ellipticity of the operator $F$ can be reformulated as
\begin{equation}\label{pucciyizhitioyuan}
	P_{\lambda,\Lambda}^{-}(N)\leq F(M+N,x)-F(M,x)\leq P_{\lambda,\Lambda}^{+}(N)
\end{equation}
for all $M,N\in S^{d}$ and $x\in B_{1}$.

In the sequel, for the reader's convenience, we shall present the notion of viscosity solution for the following equations
\begin{equation}\label{model22}
	G(D^{2}u,Du,x):=f(x)-\Phi(\abs{Du}, x)F(D^2 u, x)-H(Du,x)=0 \quad  \text{in} \quad B_{1},
\end{equation}
which was introduced in \cite[Definition 2.7]{Birindelli2004} and \cite[Definition 2.1]{Birindelli2010JDE}.
\begin{definition}
	A function $u\in C(B_{1})$ is a viscosity supersolution (resp. subsolution) to \eqref{model22}, if for every $x_0 \in B_{1}$ either there exists $\delta>0$ such that $u$ is constant in $B_{\delta}(x_{0})$ and $f(x)\geq 0$ (resp. $f(x)\leq 0$) for all $x\in B_{\delta}(x_{0})$, or, for all $\varphi \in C^2\left(B_{1}\right)$ such that $ u -\varphi$ attains a local minimum (resp. local maximum) at $x_0$ and $D\varphi(x_{0})\neq 0$, it holds  
	$$
	G(D^2\varphi(x_0),D\varphi(x_0),x_0) \geq  0 \quad({\rm resp.}\; G(D^2\varphi(x_0),D\varphi(x_0),x_0) \leq  0).
	$$
	Finally, a function $u\in C(B_{1})$ is said to be a viscosity solution of \eqref{model22} if it is simultaneously a viscosity supersolution and a viscosity subsolution. 
\end{definition}

We conclude this section with the concepts of non-collapsing sets and shored-up sequences, as well as two known lemmas from \cite{Andrade2022}. These are pivotal for proving our main theorems.

\begin{definition}[\bf Non-collapsing]
	A set $\Gamma$ of moduli of continuity defined over an interval $I\in \mathcal{I}:=\left\{(0,T]|0<T<\infty\right\}\cup \left\{\mathbb{R}_{0}^{+}\right\}$ is said to be non-collapsing, whenever for any sequence $\left\{f_j\right\}\subset\Gamma$ and any sequence $\left\{a_{j}\right\}\subset I$, we have 
	 $$f_j(a_j)\rightarrow 0 \quad {\rm implies\quad}  a_j\rightarrow0.$$
\end{definition}
\begin{definition}[\bf Shored-up]\label{def2.5}
	A sequence of moduli of continuity $\left\{\gamma_j\right\}$ is said to be shored-up, if there exists a sequence of positive numbers $\{a_j\}$ with $a_j\rightarrow0$ such that
	$$\inf_j\gamma_j(a_j)>0\quad {\rm for \;each \;} j\in\mathbb{N}.$$
\end{definition}
\begin{lemma}\label{lem2.6} {\rm(cf. \cite[Lemma 1]{Andrade2022})} 
	Let $\epsilon,\sigma>0$ and $\left\{a_j\right\}\in \ell^{1}$. Then there exists a positive sequence $\left\{c_j\right\}$ satisfying 
	$\lim\limits_{j\rightarrow \infty} c_{j}=0$ and $\max\limits_j{|c_j|}\le\frac{1}{\epsilon}$,
	such that
	\begin{equation*}
	\left\{\frac{a_{j}}{c_{j}}\right\}\in\ell^1 \quad {\rm and}\quad 	\epsilon\left(1-\frac{\sigma}{2}\right)\sum_{j=1}^{\infty}a_{j}\leq \sum_{j=1}^{\infty} \frac{a_j}{c_j}\leq \epsilon(1+\sigma)\sum_{j=1}^{\infty}a_{j}.
	\end{equation*}
\end{lemma}
\begin{lemma}\label{lem2.7}{\rm(cf. \cite[Proposition 5]{Andrade2022})} 
	If a sequence of moduli of continuity $\{\gamma_j\}$ is shored-up, then $\Gamma=\cup_{j\in\mathbb{N}}\{\gamma_j\}$ is non-collapsing.
\end{lemma}
\section{$C^{1}$ regularity to \eqref{4doublephasezhufangcheng} with $\Phi$ of the form \eqref{phiform1}}\label{section3}
This section is devoted to the proof of Theorem \ref{main} on the borderline gradient continuity of solutions to \eqref{4doublephasezhufangcheng} with $\Phi$ of the form \eqref{phiform1}. 
\subsection{H\"{o}lder continuity to perturbed equations}
In this subsection, we establish a compactness result for
viscosity solutions to the following perturbed equations
\begin{equation}\label{4raodongfangcheng}
	\Phi(|Du+\xi|,x)F(D^2 u, x)+H(Du+\xi,x) =f(x) \quad  \text{in} \quad B_{1},
\end{equation}
where $\xi$ is an arbitrarily vector in $\rn$. The reasoning employed in this subsection is inspired by the
methods put forward in \cite{Imbert1}. However, unlike \cite{Imbert1}, the presence of the double phase gradient term and Hamiltonian term requires careful handling of $|Du+\xi|$. 
Moreover, we need to introduce a new smooth auxiliary function tailored to our setting of variable coefficients.

\begin{proposition}\label{4aajinxingjieguo2}
	Suppose that the assumptions \eqref{A1}-\eqref{A6} and \eqref{phiform1} hold.
	Let $\xi \in \rn$ be an arbitrarily vector and $u\in C(B_{1})$ be a viscosity solution to \eqref{4raodongfangcheng} with $\|u\|_{L^{\infty}(B_{1})}\leq 1$.
	Then there exists a constant $\mathcal{J}_{0}=\mathcal{J}_{0}(d,\lambda,\Lambda,C_{F},\theta,\sigma_{1}(1),\mathcal{M},\|f\|_{L^{\infty}(B_{1})})>1$ such that
	\begin{itemize}
		\item   [{\rm$({{\rm i}})$}] if $\abs{\xi}\geq  \mathcal{J}_{0}$, then $u\in C_{\rm loc}^{0,1}(B_{1})$ with the estimate
		$$[u]_{C^{0,1}(B_{1/2})}\leq C(d,\lambda,\Lambda,C_{F},\theta,\sigma_{1}(1),\mathcal{M},\|f\|_{L^{\infty}(B_{1})}).$$
		\item [{\rm$({{\rm ii}})$}] If $\abs{\xi}< \mathcal{J}_{0}$, then
		$u\in C_{\rm loc}^{0,\tau}(B_{1})$ for some $\tau\in (0,1)$ with the estimate  
		$$[u]_{C^{0,\tau}(B_{1/2})}\leq C(d,\lambda,\Lambda,C_{F},\theta,\sigma_{1}(1),\mathcal{M},\|f\|_{L^{\infty}(B_{1})}).$$
	\end{itemize}
\end{proposition}
\begin{proof}
	(i) To prove the local Lipschitz regularity of $u$, let us fix $0<r<1$ and consider the quantity
	\begin{equation*}
		\mathcal{G}(x_{0}):=\sup\limits_{(x,y)\in B_{r}\times B_{r}}\left\{u(x)-u(y)-L_{1}w(\abs{x-y})-L_{2}\Big(\lvert x-x_{0}\rvert^{2}+\lvert y-x_{0} \rvert^{2} \Big)\right\}
	\end{equation*}
	for each $x_{0}\in B_{r/2}$, where constants $L_{1},L_{2}>1$ and
	$$
	w(s)=\left\{
	\begin{array}{lcl}
		s-w_{0}s^{1+\beta}& \text{if} & 0\leq s\leq s_{0}:=\left(\frac{1}{(1+\beta)w_{0}}\right)^{1/\beta}, \\
		w(s_{0}) &\text{if}&  s>s_{0},
	\end{array}
	\right.
	$$
	with $\beta\in(0,\theta)$ and $w_{0}\in\left(0,\frac{1}{1+\beta}\right)$. Notice that $s_{0}\geq 1$ and 
	$$w(s)\geq 0,\quad 0\leq w^{\prime}(s)\leq 1,\quad w^{\prime\prime}(s)\leq 0\quad {\rm for\; any\;} s\geq 0.$$ 
	If we prove that there exist two constants $L_{1}, L_{2} >1$ such that
	\begin{equation}\label{goal}
		\mathcal{G}(x_{0})\leq 0
	\end{equation}
	for all $x_{0} \in B_{r/2}$, then the desired result is established. We will argue by contradiction; for this reason, suppose that for all $L_{1},L_{2}>1$, there is $\hat{x}_{0}\in B_{r/2}$ for which $\mathcal{G}(\hat{x}_{0})>0$.
	
	We introduce two auxiliary functions $\psi,\Psi:\overline{B}_{r}\times\overline{B}_{r}\rightarrow \mathbb{R}$, defined by
	\begin{equation*}
		\begin{cases}
			\psi(x,y):=L_{1}w(\abs{x-y})+L_{2}\Big(\lvert x-\hat{x}_{0}\rvert^{2}+\lvert y-\hat{x}_{0} \rvert^{2} \Big),\\
			\Psi\left(x,y\right):=u(x)-u(y)-\psi(x,y).
		\end{cases}
	\end{equation*}
Let $\left(\hat{x},\hat{y}\right)\in \overline{B}_{r}\times\overline{B}_{r}$ be a maximum point of $\Psi(x,y)$, i.e., $\Psi\left(\hat{x},\hat{y}\right)= \mathcal{G}(\hat{x}_{0})>0$. We first observe that $\hat{x}\neq\hat{y}$, otherwise the maximum of $\Psi$ would be nonpositive. Then it follows from $\|u\|_{L^{\infty}(B_{1})}\leq 1$ that
	\begin{equation}\label{31}
		L_{1}w(\abs{x-y})+L_{2}\Big(\lvert \hat{x}-\hat{x}_{0}\rvert^{2}+\lvert \hat{y}-\hat{x}_{0} \rvert^{2} \Big)<u(\hat{x})-u(\hat{y})\leq 2\|u\|_{L^{\infty}(B_{1})}\leq 2.
	\end{equation}
	 Now we choose $L_{2}\geq \frac{32}{r^{2}}$. With such choice, it follows from \eqref{31} that
	\begin{equation}\label{neidian}
		\lvert \hat{x}-\hat{x}_{0}\rvert,\lvert \hat{y}-\hat{x}_{0} \rvert\leq \frac{r}{4},\quad \abs{\hat{x}-\hat{y}}\leq \sqrt{2\left(\lvert \hat{x}-\hat{x}_{0}\rvert^{2}+\lvert \hat{y}-\hat{x}_{0} \rvert^{2} \right)}\leq \frac{2}{\sqrt{L_{2}}}<1.
	\end{equation}
	This along with $\hat{x}_{0}\in B_{r/2}$ implies that $(\hat{x},\hat{y})$ is interior point of $B_{r}\times B_{r}$. 		

 Next, we invoke the Crandall-Ishii-Lions lemma (see \cite[Theorem 3.2]{Crandle1}, \cite[Proposition 2.1]{Fili}) to assure
the existence of a limiting subjet $\left(\xi_{\hat{x}},X\right)$ of $u$ at $\hat{x}$ and a limiting superjet $\left(\xi_{\hat{y}},Y\right)$ of $u$ at $\hat{y}$, where
\begin{equation}\label{4aasubsuperjet}
	\xi_{\hat{x}}:=L_{1}w^{\prime}(\abs{\hat{x}-\hat{y}})\frac{\hat{x}-\hat{y}}{\abs{\hat{x}-\hat{y}}}+2L_{2}(\hat{x}-\hat{x}_{0}), \quad  \xi_{\hat{y}}:=L_{1}w^{\prime}(\abs{\hat{x}-\hat{y}})\frac{\hat{x}-\hat{y}}{\abs{\hat{x}-\hat{y}}}-2L_{2}(\hat{y}-\hat{x}_{0})
\end{equation}
such that the matrices $X,Y\in S^{d}$ satisfy the matrix inequality
	\begin{equation}\label{matrix2}
		\left(
		\begin{array}{cc}
			X & 0 \\
			0 & -Y \\
		\end{array}
		\right)
		\leq  \left(
		\begin{array}{cc}
			A & -A \\
			-A & A \\
		\end{array}
		\right)+
		(2L_{2}+\epsilon)
		\left(
		\begin{array}{cc}
			I & 0 \\
			0 & I \\
		\end{array}
		\right)
	\end{equation}
for
\begin{equation}\label{4aajuzhena}
	A:=L_{1}\left[\frac{w^{\prime}(\abs{\hat{x}-\hat{y}})}{\abs{\hat{x}-\hat{y}}}I+\left(w^{\prime\prime}(\abs{\hat{x}-\hat{y}})-\frac{w^{\prime}(\abs{\hat{x}-\hat{y}})}{\abs{\hat{x}-\hat{y}}}\right)\frac{(\hat{x}-\hat{y})\otimes(\hat{x}-\hat{y})}{\abs{\hat{x}-\hat{y}}^{2}}\right]
\end{equation}
and $\epsilon\in (0,1)$, that only depends on the norm of $A$ and can be made sufficiently small. Furthermore, we have the following viscosity inequalities
	\begin{equation}\label{solution1}
		\Phi(\abs{\xi_{\hat{x}}+\xi},\hat{x})F(X, \hat{x})+H(\xi_{\hat{x}}+\xi,\hat{x})\geq f(\hat{x}),
	\end{equation}
	\begin{equation}\label{solution2}
			\Phi(\abs{\xi_{\hat{y}}+\xi},\hat{y})F(Y,\hat{y})+H(\xi_{\hat{y}}+\xi,\hat{y})\leq f(\hat{y}).
	\end{equation}
Utilizing the uniform ellipticity of operator $F$ at the point $\hat{x}$ and $\eqref{A2}$ to obtain
\begin{equation}\label{uniformx}
	\begin{split}
		F(X,\hat{x})\leq& P_{\lambda,\Lambda}^{+}(X-Y)+F(Y,\hat{y})+F(Y,\hat{x})-F(Y,\hat{y})\\
		\leq& P_{\lambda,\Lambda}^{+}(X-Y)+F(Y,\hat{y})+C_{F}\|Y\|\abs{\hat{x}-\hat{y}}^{\theta}.
	\end{split}
\end{equation}
Combining \eqref{uniformx} with \eqref{solution1} and \eqref{solution2}, we arrive at
\begin{equation}\label{bds}
	J:=\frac{f(\hat{x})-H(\xi_{\hat{x}}+\xi,\hat{x})}{\Phi(\abs{\xi_{\hat{x}}+\xi},\hat{x})}-\frac{f(\hat{y})-H(\xi_{\hat{y}}+\xi,\hat{y})}{\Phi(\abs{\xi_{\hat{y}}+\xi},\hat{y})}
	\leq P_{\lambda,\Lambda}^{+}(X-Y) +C_{F}\|Y\|\abs{\hat{x}-\hat{y}}^{\theta}.
\end{equation}
{\bf Esimate of $P_{\lambda,\Lambda}^{+}(X-Y)$}. We employ the matrix inequality \eqref{matrix2} to vectors of the form $(z,z)\in\mathbb{R}^{2d}$ with $\abs{z}=1$ and obtain
	\begin{equation}\label{superasati}
		\left\langle(X-Y)z,z\right\rangle\leq 4L_{2}+2\epsilon.
	\end{equation}
	This means that all the eigenvalues of $X-Y$ are less than or equal to $4L_{2}+2\epsilon$. In particular, applying \eqref{matrix2} to vector of the form $(\upsilon,-\upsilon)\in\mathbb{R}^{2d}$, where $\upsilon$ is given by $\upsilon:=\frac{\hat{x}-\hat{y}}{\abs{\hat{x}-\hat{y}}}$, we have 
	\begin{equation*}
		\left\langle(X-Y)\upsilon,\upsilon\right\rangle\leq 4L_{1}w^{\prime\prime}(\abs{\hat{x}-\hat{y}})+4L_{2}+2\epsilon.
	\end{equation*}	
	This yields that
	at least one eigenvalue of $X-Y$ is less than $4L_{1}w^{\prime\prime}(\abs{\hat{x}-\hat{y}})+4L_{2}+2\epsilon$. We choose $L_{1}>\frac{2L_{2}+1}{2\beta(1+\beta)w_{0}}$, it follows from $\abs{\hat{x}-\hat{y}}<1$ and $\beta<\theta<1$ that
	\begin{align*}
		4L_{1}w^{\prime\prime}(\abs{\hat{x}-\hat{y}})+4L_{2}+2\epsilon&=-4L_{1}\beta(\beta+1)w_{0}\abs{\hat{x}-\hat{y}}^{\beta-1}+4L_{2}+2\epsilon\\
		&<-4L_{1}\beta(\beta+1)w_{0}+4L_{2}+2<0.
	\end{align*}
	This means that at least one eigenvalue of
	$X-Y$ is negative. Based on the above analysis and the definition of Pucci extremal operator, we arrive at
	\begin{equation}\label{puccishangjie}
		P_{\lambda,\Lambda}^{+}(X-Y)\leq \left(\lambda+\Lambda(d-1)\right)(4L_{2}+2\epsilon)-4L_{1}\lambda\beta(\beta+1)w_{0}\abs{\hat{x}-\hat{y}}^{\beta-1}.
	\end{equation}
	{\bf Esimate of $\|Y\|$}. 
	Applying \eqref{matrix2} to the vectors $(0,z)\in\mathbb{R}^{2d}$ with $\abs{z}=1$,
	we get
	\begin{equation}\label{4gujifanshuuse}
		\left\langle-Yz,z\right\rangle\leq \left\langle Az,z\right\rangle+2L_{2}+\epsilon.
	\end{equation}
It follows from the definition of the matrix $A$ in \eqref{4aajuzhena} that $$A\leq L_{1}\frac{w^{\prime}(\abs{\hat{x}-\hat{y}})}{\abs{\hat{x}-\hat{y}}}I.$$ 
This together with \eqref{4gujifanshuuse} and $0\leq w^{\prime}(s)\leq 1$ for $s\geq 0$ yields that
	\begin{equation}\label{coeficentshangjie}
		\|Y\|\leq L_{1}\abs{\hat{x}-\hat{y}}^{-1}+2L_{2}+\epsilon.
	\end{equation}
{\bf Esimate of $J$}. Set $\mathcal{J}_{0}:=4L_{1}>1$ for $L_{1}$ to be fixed later and suppose $\abs{\xi}\geq \mathcal{J}_{0}$. 
We choose $L_{1}>L_{2}$, then it follows from $\Abs{\hat{x}-x_{0}}< \frac{1}{4}$ and $0\leq w^{\prime}(s)\leq 1$ for $s\geq 0$ that
	\begin{equation*}
		\abs{\xi_{\hat{x}}}\leq L_{1}w^{\prime}(\abs{\hat{x}-\hat{y}})+2L_{2}\Abs{\hat{x}-x_{0}}< 2L_{1}.
	\end{equation*}
	Then it follows that
	\begin{align*}
		\abs{\xi_{\hat{x}}+\xi}\geq\abs{\xi}- \abs{\xi_{\hat{x}}}\geq \mathcal{J}_{0}-2L_{1}=2L_{1}>1.
	\end{align*}
	In exactly the same way, we get
	\begin{equation*}
		\abs{\xi_{\hat{y}}+\xi}\geq 2L_{1}> 1.
	\end{equation*}
	Since $\sigma_{1}$ and $\sigma_{2}$ are nondecreasing functions, combining the last two displays with \eqref{A4}, we arrive at 
	\begin{equation}\label{yjie}
		\Phi(\abs{\xi_{\hat{x}}+\xi},\hat{x})\geq \sigma_{1}(\abs{\xi_{\hat{x}}+\xi})\geq \sigma_{1}(1),\quad 
		\Phi(\abs{\xi_{\hat{y}}+\xi},\hat{y})\geq \sigma_{1}(\abs{\xi_{\hat{y}}+\xi})\geq \sigma_{1}(1).
	\end{equation}
	A combination of \eqref{yjie} with the assumption \eqref{A5} yields that
	\begin{equation*}
		\begin{split}
			J&\geq -\frac{\|f\|_{L^{\infty}(B_{1})}+\mathcal{M}\left(1+\sigma_{1}(\abs{\xi_{\hat{x}}+\xi})\right)}{\sigma_{1}(\abs{\xi_{\hat{x}}+\xi})}-\frac{\|f\|_{L^{\infty}(B_{1})}+\mathcal{M}\left(1+\sigma_{1}(\abs{\xi_{\hat{y}}+\xi})\right)}{\sigma_{1}(\abs{\xi_{\hat{y}}+\xi})}\\
			&\geq -2\mathcal{M}-\frac{2(1+\|f\|_{L^{\infty}(B_{1})})}{\sigma_{1}(1)}. 
		\end{split}
	\end{equation*}
	Substituting the above estimate, \eqref{puccishangjie} and \eqref{coeficentshangjie}  into \eqref{bds}, and using the fact $\abs{\hat{x}-\hat{y}}^{\theta}<1$, we get
	\begin{equation*}
		-2\mathcal{M}-\frac{2(1+\|f\|_{L^{\infty}(B_{1})})}{\sigma_{1}(1)}
		\leq C_{1}+
		L_{1}\abs{\hat{x}-\hat{y}}^{\beta-1}\left(C_{F}\abs{\hat{x}-\hat{y}}^{\theta-\beta}-4\lambda w_{0}\beta(1+\beta)\right),
	\end{equation*}
wrehe $C_{1}:=\left(\Lambda(d-1)+\lambda\right)(4L_{2}+2)+C_{F}(2L_{2}+1)$.
Selecting $L_{2}\geq 4\left(\frac{C_{F}}{2\lambda w_{0}\beta(1+\beta)}\right)^{2/(\theta-\beta)}$. We exploit \eqref{neidian} and $\beta<\theta$ to derive
	\begin{equation*}
		C_{F}\Abs{\hat{x}-\hat{y}}^{\theta-\beta}\leq C_{F}\left(\frac{2}{\sqrt{L_{2}}}\right)^{\theta-\beta}\leq 2\lambda w_{0}\beta(1+\beta).
	\end{equation*}
	This along with  $\abs{\hat{x}-\hat{y}}^{\beta-1}>1$ yields that
	\begin{equation}\label{44contradiction}
		-2\mathcal{M}-\frac{2}{\sigma_{1}(1)}(1+\|f\|_{L^{\infty}(B_{1})})<C_{1}-2L_{1}\lambda w_{0}\beta(1+\beta).
	\end{equation}
	As a consequence, by choosing $L_{1}\geq \frac{C_{1}+2\mathcal{M}}{2\lambda w_{0}\beta(1+\beta)}+\frac{1+\|f\|_{L^{\infty}(B_{1})}}{\lambda w_{0}\beta(1+\beta)\sigma_{1}(1)}$, we reach a contradiction with \eqref{44contradiction}.\\
	(ii) Suppose $\abs{\xi}< \mathcal{J}_{0}$ with $\mathcal{J}_{0}$ as in (i). Let $p\in\mathbb{R}^{d}$ such that $\abs{p}\geq 3\mathcal{J}_{0}$. Then it follows that $\abs{p+\xi}\geq 2\mathcal{J}_{0}>1$, which together with \eqref{A3}-\eqref{A5} yields that  $$\Phi(\abs{p+\xi},x)=\sigma_{1}(\abs{p+\xi})+a(x)\sigma_{2}(\abs{p+\xi})\geq \sigma_{1}(1),$$
	$$\Abs{\frac{{H({p+\xi},x)}}{\Phi(\abs{p+\xi},x)}}\leq \frac{\mathcal{M}(1+\sigma_{1}(\abs{p+\xi}))}{\sigma_{1}(\abs{p+\xi})}\leq \mathcal{M}\left(1+\frac{1}{\sigma_{1}(1)}\right).$$
	Therefore, we can see that $u$ is a viscosity solution of
	\begin{equation*}
		\begin{cases}
			P^{+}_{\lambda,\Lambda}(D^{2}u)+\mathcal{M}\left(1+\frac{1}{\sigma_{1}(1)}\right)+\frac{\|f\|_{L^{\infty}(B_{1})}}{\sigma_{1}(1)}\geq 0\quad\quad {\rm in} \;\; B_{1}\cap\left\{|Du|\geq 3\mathcal{J}_{0}\right\}, \\
			P^{-}_{\lambda,\Lambda}(D^{2}u)-\mathcal{M}\left(1+\frac{1}{\sigma_{1}(1)}\right)-\frac{\|f\|_{L^{\infty}(B_{1})}}{\sigma_{1}(1)}\leq 0 \quad\quad {\rm in} \;\; B_{1}\cap\left\{|Du|\geq 3\mathcal{J}_{0}\right\}.
		\end{cases}
	\end{equation*}	
		At this point, we are in an exact position to apply \cite[Theorem 1.1]{Silvestre2016JEMS} to know that $u$ is local H\"{o}lder continuous. This completes the proof.	
\end{proof}
\subsection{Tangential analysis}
The compactness stemming from the former result unlocks a new approximation lemma, instrumental in our analysis. This is the content of the following lemma.
\begin{lemma}[{\bf Approximation Lemma}]\label{lem4.1}
	 Suppose that the assumptions \eqref{A1}-\eqref{A6} and \eqref{phiform1} hold. Let $\sigma_{1}+a(\cdot)\sigma_{2}\in\Gamma$, where
	 $\Gamma$ is a collection of non-collapsing moduli of continuity. Let $\xi\in \rn$ be an arbitrarily vector and  $u\in C(B_{1})$ be a viscosity solution of to \eqref{4raodongfangcheng} with $\|u\|_{L^{\infty}(B_{1})}\leq 1$.
For any $\varepsilon>0$, there exists $\delta=\delta(d,\lambda,\Lambda,\varepsilon,\Gamma)\in (0,1)$ such that if
	\begin{equation*}
	\max\left\{\|{\rm osc}_{{F}}\|_{L^{\infty}\left(B_{1}\right)},\|f\|_{L^{\infty}(B_{1})},\mathcal{M} \right\}\leq \delta,
	\end{equation*}
then there exists a $F$-harmonic function $h\in C^{1,\alpha_{0}}_{\rm loc}(B_{3/4})$ (i.e., $h$ is a viscosity solution of $F(D^{2}h)=0$) such that
	 \begin{equation*}
	 	\|u-h\|_{L^{\infty}(B_{1/2})}\leq \varepsilon.
	 \end{equation*}
\end{lemma}
\begin{proof}
The proof resorts to a contradiction argument. Suppose that the result does not hold. Then there exist $\varepsilon_{0}>0$ and sequences  $\{F_{j}\}_{j\in \mathbb{N}}$, $\{\Phi_{j}\}_{j\in \mathbb{N}}$, $\{H_{j}\}_{j\in \mathbb{N}}$, $\{\xi_{j}\}_{j\in \mathbb{N}}$, $\{f_{j}\}_{j\in \mathbb{N}}$, $\{u_{j}\}_{j\in \mathbb{N}}$ such that
	\begin{itemize}
		\item   [{\rm$({{\rm i}})$}] the operator $F_{j}:S^{d}\times B_{1}\rightarrow \mathbb{R}$ is uniformly $(\lambda,\Lambda)$-elliptic with $\|{\rm osc}_{{F_{j}}}\|_{L^{\infty}\left(B_{1}\right)}\leq \frac{1}{j}$;
		\item [{\rm$({{\rm ii}})$}] $\Phi_{j}(t,x):=\sigma_{1}^{j}(t)+a_{j}(x)\sigma_{2}^{j}(t)$, where $\sigma_{i}^{j}(\cdot)$ ($i=1,2$) are moduli of continuity, and $0\leq a_{j}(\cdot)\in C(B_{1})$. Moreover, if $\sigma_{1}^{j}(b_{j})+a_{j}(x)\sigma_{2}^{j}(b_{j})\rightarrow 0$, then $b_{j}\rightarrow 0$;
		\item [ {\rm$({{\rm iii}})$}] $f_{j}\in C({B_{1}})$ with $\|f_{j}\|_{L^{\infty}(B_{1})}\leq\frac{1}{j}$;
		\item [ {\rm$({{\rm iv}})$}] the Hamiltonian term fulfills $\Abs{H_{j}(t,x)}\leq \mathcal{M}_{j}(1+\sigma_{1}^{j}(\abs{t}))$ with $\mathcal{M}_{j}\leq\frac{1}{j}$;
		\item [ {\rm$({{\rm v}})$}] $u_{j}\in C({B_{1}})$ with $ \|u_{j}\|_{L^{\infty}(B_{1})}\leq 1$ is a viscosity solution of
		\begin{equation}\label{model333}
			\Phi_{j}(|Du_{j}+\xi_{j}|,x)F_{j}(D^2 u_{j}, x)+H_{j}(Du_{j}+\xi_{j},x) =f_{j}(x) \quad  \text{in} \quad B_{1}.
		\end{equation}
	\end{itemize}
	Nonetheless, for any $F$-harmonic function $h\in C^{1,\alpha_{0}}_{\rm loc}\left(B_{3/4}\right)$, it holds
	 \begin{equation}\label{daomaodun}
		\|u_{j}-h\|_{L^{\infty}(B_{1/2})}>\varepsilon_{0}\quad {\rm for\; any \;} j\in \mathbb{N}.
	\end{equation}
	
	Since ${F_{j}}$ are uniformly $(\lambda,\Lambda)$-elliptic, they are also Lipschitz continuous in $M$. Thus, it follows from (i) and Arzel${\rm \grave{a}}$-Ascoli theorem that there exists some uniformly $(\lambda,\Lambda)$-elliptic
	operator $F_{\infty}$ (with frozen coefficients) such that $F_{j}\rightarrow F_{\infty}$ locally uniformly in $S^{d}$ for all $x\in B_{1}$ fixed, through a subsequence if necessary. In addition, we know from Proposition \ref{4aajinxingjieguo2} that the sequence $\{u_{j}\}_{j\in\mathbb{N}}\subset C_{\rm loc}^{0,\tau}(B_{1})$ for some $\tau\in (0,1)$. Therefore, by applying Arzel${\rm \grave{a}}$-Ascoli theorem again, we conclude that, up to a subsequence,  $u_{j}$ converges locally uniformly in $B_{1}$ to some continuous function $u_{\infty}$ in the $C^{0}$-topology. 

In the sequel, our goal is to verify that the limiting function $u_{\infty}$ is a viscosity solution to the homogeneous equation
	 \begin{equation}\label{homojie}
		F_{\infty}(D^{2}v)=0 \quad {\rm in}\quad  B_{3/4}.
	\end{equation}
We only show that $u_{\infty}$ is a viscosity supersolution, as its subsolution counterpart is entirely analogous. Let $\varphi$ be any test function touching $u_{\infty}$ from below at a point $\overline{x}\in B_{3/4}$, that is,
	$$\varphi(\overline{x})=u_{\infty}(\overline{x})\quad {\rm and}\quad \varphi(x)<u_{\infty}(x)\quad {\rm for\; all}\;\;x\neq \overline{x}.$$
	Without loss of generality, we assume that $\abs{\overline{x}}=u_{\infty}(\overline{x}) = 0$ and $\varphi$ is a quadratic polynomial, namely,
	$$\varphi(x)=\frac{1}{2}\left\langle Mx,x\right\rangle+\left\langle b,x\right\rangle.$$
	Since $u_{j}\rightarrow u_{\infty}$ locally uniformly in $B_{1}$, we see that, for $j$ sufficiently large, the polynomial
	$$\varphi_{j}(x):=\frac{1}{2}\left\langle M(x-x_{j}),x-x_{j}\right\rangle+\left\langle b,x-x_{j}\right\rangle+u_{j}(x_{j})$$
	touches $u_{j}$ from below at $x_{j}$ belonging to a small neighbourhood of zero. Since $u_{j}$ is a viscosity solution of \eqref{model333}, we immediately obtain that
	\begin{equation}\label{sj}
		\Phi_{j}(\abs{b+\xi_{j}},x_{j})F_{j}(M, x_{j})+H_{j}(b+\xi_{j},x_{j}) \leq f_{j}(x_{j}).
	\end{equation}

If the sequence $\{\xi_{j}\}_{j\in \mathbb{N}}$ is unbounded, then we may assume $\abs{\xi_{j}}\rightarrow {\infty}$ as $j\rightarrow\infty$ (up to a subsequence). As a result, there exists $j^{\star}\in\mathbb{N}$ so large that $\abs{\xi_{j}}\geq2\max\{1,\abs{b}\}$ for all $j\geq j^{\star}$. By the triangle inequality, we get
\begin{equation*}
	\abs{b+\xi_{j}}\geq \abs{\xi_{j}}-|b|\geq \frac{1}{2}\abs{\xi_{j}}\geq 1.
\end{equation*}
This along with the assumptions (ii)-(iv) yields that
\begin{equation}\label{f1}
	\frac{f_{j}(x_{j})}{\Phi_{j}(\abs{b+\xi_{j}},x_{j})}\leq  \frac{\|f_{j}\|_{L^{\infty}(B_{1})}}{\sigma_{1}^{j}(1)}\leq \frac{1}{j\sigma_{1}^{j}(1)},
\end{equation}
\begin{equation}\label{hjie}
		\frac{H_{j}\left(b+\xi_{j},x_{j}\right)}{\Phi_{j}(\abs{b+\xi_{j}},x_{j})}\leq \frac{\mathcal{M}_{j}\left(1+\sigma_{1}^{j}(\abs{b+\xi_{j}})\right)}{\sigma_{1}^{j}(\abs{b+\xi_{j}})}\leq \frac{1}{j}\left(1+\frac{1}{\sigma_{1}^{j}(1)}\right).
\end{equation}
A combination of \eqref{sj} with \eqref{f1} and \eqref{hjie} leads to that
\begin{equation*}
	\begin{split}
		F_{\infty}(M)=\lim\limits_{j\rightarrow \infty}F_{j}(M,x_{j})
		\leq \lim\limits_{j\rightarrow \infty}\left(\Abs{\frac{f_{j}(x_{j})}{\Phi_{j}(\abs{b+\xi_{j}},x_{j})}}+ \Abs{\frac{H_{j}\left(b+\xi_{j},x_{j}\right)}{\Phi_{j}(\abs{b+\xi_{j}},x_{j})}}\right)= 0.
	\end{split}
\end{equation*}

On the other hand, if the sequence $\{\xi_{j}\}_{j\in\mathbb{N}}$ is bounded, then we may assume  $\xi_{j}\rightarrow \xi_{\infty}$ as $j\rightarrow \infty$ (up to a subsequence). As a
consequence, $\xi_{j}+b\rightarrow \xi_{\infty}+b \; {\rm as}\; j\rightarrow \infty.$ At this point, we consider two cases: $\abs{b+\xi_{\infty}}\neq 0$ or $\abs{b+\xi_{\infty}}= 0$. 

In the case of $\abs{b+\xi_{\infty}}>0$. According to (ii), we know that $\Phi(\abs{b+\xi_{j}},x_{j})\nrightarrow 0$ as $j\rightarrow \infty$. Then it follows from $\|f_{j}\|_{L^{\infty}(B_{1})}\leq\frac{1}{j}$ and $\mathcal{M}_{j}\leq \frac{1}{j}$ that
\begin{equation*}
	F_{\infty}(M)=\lim\limits_{j\rightarrow \infty}F_{j}(M,x_{j})\leq\lim\limits_{j\rightarrow \infty}\left(\Abs{\frac{f_{j}(x_{j})}{\Phi_{j}(\abs{b+\xi_{j}},x_{j})}}+ \Abs{\frac{H_{j}\left(b+\xi_{j},x_{j}\right)}{\Phi_{j}(\abs{b+\xi_{j}},x_{j})}}\right)= 0.
\end{equation*}

For the latter case where $\Abs{b+\xi_{\infty}}= 0$, there are two possibilities, namely, $b=-\xi_{\infty}$ with $|b|,\abs{\xi_{\infty}}>0$ or $|b|=\abs{\xi_{\infty}}=0$. We are going to justify $F_{\infty}(M)\leq 0$. By contradiction, let us assume that
\begin{equation}\label{contracdiction}
	F_{\infty}(M)>0.
\end{equation}
The ellipticity condition of $F_{\infty}$ implies that matrix $M$ has at least one positive eigenvalue. Let $\rn=T\oplus Q$ be the orthogonal sum, where $T:={\rm span}\{e_{1},e_{2},...,e_{k}\}$ is the invariant space composed of those eigenvectors corresponding to
positive eigenvalues of $M$, and $Q:=\{y\in \rn:\left\langle y,\eta  \right\rangle=0\;{\rm for\;all}\;\eta\in T\}$. 

{\bf Case 1.} $b=-\xi_{\infty}$ with $|b|,\abs{\xi_{\infty}}>0$.
Let $\gamma>0$ and
\begin{equation}\label{model3}
	\varphi_{\gamma}(x):=\varphi(x)+\gamma \Abs{P_{T}(x)}=\frac{1}{2}\left\langle Mx,x\right\rangle+\left\langle b,x\right\rangle+\gamma \Abs{P_{T}(x)},
\end{equation}
where $P_{T}$ stands for the orthogonal projection over $T$. Since $u_{j}\rightarrow u_{\infty}$ locally uniformly in $B_{1}$ and $\varphi$ touches $u_{\infty}$ from below at the origin, then, for $\gamma$ small enough, $\varphi_{\gamma}$ touches
$u_{j}$ from below at a point $x_{j}^{\gamma}$ belonging to a small neighbourhood of 0. Moreover, there holds that, up to a subsequence, $x_{j}^{\gamma}\rightarrow x_{*}$ for some $x_{*}\in B_{3/4}$ as $j\rightarrow\infty$. Now we treat separately the
cases where $P_{T}\left(x_{j}^{\gamma}\right)=0$ and $P_{T}\left(x_{j}^{\gamma}\right)\neq 0$.

First, we consider $P_{T}\left(x_{j}^{\gamma}\right)=0$.
Notice that
\begin{equation*}
	\tilde{\varphi}_{\gamma}(x):=\varphi(x)+\gamma \left\langle e,P_{T}(x)\right\rangle
\end{equation*}
touches $u_{j}$ from below at $x_{j}^{\gamma}$ for every $e\in \mathbb{S}^{d-1}$ (i.e., $|e|=1$). Through a direct calculation, we derive
\begin{equation*}
	D\tilde{\varphi}_{\gamma}(x_{j}^{\gamma})=Mx_{j}^{\gamma}+b+\gamma P_{T}(e),\quad D^{2}\tilde{\varphi}_{\gamma}(x_{j}^{\gamma})=M.
\end{equation*}
We choose $e\in T\cap \mathbb{S}^{d-1}$ such that $P_{T}(e)=e$.  Since $u_{j}$ is a viscosity solution of \eqref{model333}, we get
\begin{equation}\label{model8}
	\Phi_{j}(\abs{Mx_{j}^{\gamma}+b+\gamma e+\xi_{j}},x_{j}^{\gamma})F_{j}(M, x_{j}^{\gamma})+H_{j}\left(Mx_{j}^{\gamma}+b+\gamma e+\xi_{j},x_{j}^{\gamma}\right) \leq f_{j}(x_{j}^{\gamma}).
\end{equation}
If $Mx_{*}=0$, then for $j$ sufficiently large, we have
\begin{equation*}
	\Abs{Mx_{j}^{\gamma}+b+\xi_{j}}\leq \frac{\gamma}{2}.
\end{equation*}
This along with the triangle inequality yields
\begin{equation*}
	\frac{\gamma}{2}\leq \Abs{Mx_{j}^{\gamma}+b+\gamma e+\xi_{j}}\leq  \frac{3\gamma}{2}.
\end{equation*}
In light of (ii)-(iv), we deduce that
\begin{equation*}
		\frac{f_{j}(x_{j}^{\gamma})}{\Phi_{j}(\abs{Mx_{j}^{\gamma}+b+\gamma e+\xi_{j}},x_{j}^{\gamma})}\leq \frac{1}{j\sigma_{1}^{j}(\gamma/2)},
\end{equation*}
\begin{equation*}
	\frac{H_{j}\left(Mx_{j}^{\gamma}+b+\gamma e+\xi_{j},x_{j}^{\gamma}\right)}{\Phi_{j}(\abs{Mx_{j}^{\gamma}+b+\gamma e+\xi_{j}},x_{j}^{\gamma})}\leq \frac{\mathcal{M}_{j}\left(1+\sigma_{1}^{j}(\abs{Mx_{j}^{\gamma}+b+\gamma e+\xi_{j}})\right)}{\sigma_{1}^{j}(\abs{Mx_{j}^{\gamma}+b+\gamma e+\xi_{j}})}\leq\frac{1+\sigma_{1}^{j}(3\gamma/2)}{j\sigma_{1}^{j}(\gamma/2)}.
\end{equation*}
Combining the previous two inequalities with \eqref{model8}, and letting $j\rightarrow \infty$, we conclude $F_{\infty}(M)\leq 0$, which contradicts \eqref{contracdiction}.

On the other hand, if $\Abs{Mx_{*}}>0$, we start off by considering the case in which $T\equiv \rn$ and select $e\in \mathbb{S}^{d-1}$ such that
\begin{equation*}
	\Abs{Mx_{*}+\gamma P_{T}(e)}=\Abs{Mx_{*}+\gamma e}>0.
\end{equation*}
For $j$ large enough, we have
\begin{equation}\label{case1}
	\frac{1}{2}\Abs{Mx_{*}+\gamma e}\leq \Abs{Mx_{j}^{\gamma}+\gamma e}\leq \frac{3}{2}\Abs{Mx_{*}+\gamma e}\quad {\rm and}\quad \abs{\xi_{j}+b}\leq \frac{1}{8}\Abs{Mx_{*}+\gamma e}.
\end{equation}
Furthermore, if $T\neq \rn$, then we choose $e\in \mathbb{S}^{d-1}\cap T^{\perp}$ such that
\begin{equation*}
	\Abs{Mx_{*}+\gamma P_{T}(e)}=\Abs{Mx_{*}}>0.
\end{equation*}
Again for $j$ large enough, we have
\begin{equation}\label{case2}
	\frac{1}{2}\Abs{Mx_{*}}\leq \Abs{Mx_{j}^{\gamma}}\leq \frac{3}{2}\Abs{Mx_{*}}\quad {\rm and}\quad \abs{\xi_{j}+b}\leq \frac{1}{8}\Abs{Mx_{*}}.
\end{equation}
Thus, using either \eqref{case1} or \eqref{case2}, we arrive at
\begin{equation*}
	0<\frac{3}{8}\Abs{Mx_{*}+\gamma P_{T}(e)}\leq \Abs{Mx_{j}^{\gamma}+b+\gamma P_{T}(e)+\xi_{j}}\leq \frac{13}{8}\Abs{Mx_{*}+\gamma P_{T}(e)}.
\end{equation*}
By the same arguments as before, we can also conclude that $F_{\infty}(M)\leq 0$ and we reach again a contradiction to \eqref{contracdiction}.

Next, let us consider $P_{T}\left(x_{j}^{\gamma}\right)\neq 0$. Note that $\Abs{P_{T}(x)}$ is smooth and convex in a small neighbourhood of $x_{j}^{\gamma}$. Because of $P_{T}$ being a projection, then
\begin{equation}\label{feifuding}
\Abs{P_{T}(x_{j}^{\gamma})}D\left(\Abs{P_{T}(x_{j}^{\gamma})}\right)=P_{T}(x_{j}^{\gamma})\quad {\rm and} \quad D^{2}\left(\Abs{P_{T}(x_{j}^{\gamma})}\right)\;\,{\rm is \;nonnegative\; definite}.
\end{equation}
Hence, we arrive at
\begin{equation*}
		\Phi_{j}(\abs{Mx_{j}^{\gamma}+b+\gamma \zeta_{j}^{\gamma}+\xi_{j}},x_{j}^{\gamma})F_{j}(M+\gamma D^{2}\left(\Abs{P_{T}(x_{j}^{\gamma})}\right), x_{j}^{\gamma})
		+H_{j}\left(Mx_{j}^{\gamma}+b+\gamma \zeta_{j}^{\gamma}+\xi_{j},x_{j}^{\gamma}\right) \leq f_{j}(x_{j}^{\gamma}),
\end{equation*}
where 	$\zeta_{j}^{\gamma}:=\frac{{P_{T}(x_{j}^{\gamma})}}{\Abs{P_{T}(x_{j}^{\gamma})}}$.
Observe that $\Abs{\zeta_{j}^{\gamma}}=1$. Set $e:=\zeta_{j}^{\gamma}$, we can perform the same procedure as in the case $P_{T}\left(x_{j}^{\gamma}\right)= 0$ via distinguishing $Mx_{*}=0$ and $Mx_{*}\neq 0$. Then we can conclude that
$$F_{\infty}(M+\gamma D^{2}\left(\Abs{P_{T}(x_{*})}\right))\leq 0.$$
By virtue of \eqref{feifuding} and the ellipticity condition on $F_{\infty}$, we derive $$F_{\infty}(M)\leq F_{\infty}(M+\gamma D^{2}\left(\Abs{P_{T}(x_{*})}\right))\leq0,$$ which contradicts the assumption \eqref{contracdiction}.

{\bf Case 2.} $|b|=|\xi_{\infty}|=0$. In this case, the procedures become easier. Since $\frac{1}{2}\left\langle Mx,x\right\rangle$ touches $u_{\infty}$
from below at the origin and $u_{j}\rightarrow u_{\infty}$ locally uniformly, then the test function
$$\hat{\varphi}(x):=\frac{1}{2}\left\langle Mx,x\right\rangle+\gamma \Abs{P_{T}(x)}$$
touches $u_{j}$ from below at a point $\hat{x}_{j}^{\gamma}\in B_{r}$ for $\gamma>0$ sufficiently small. Likewise, we will analyze two scenarios: $\Abs{P_{T}(\hat{x}_{j}^{\gamma})}=0$ and $\Abs{P_{T}(\hat{x}_{j}^{\gamma})}\neq 0$, which is in a similar manner as above. Eventually, we also conclude that $F_{\infty}(M)\leq 0$. 

As has been stated above, we show
that $u_{\infty}$ is a viscosity solution of \eqref{homojie}. 
It follows from the regularity results in \cite[
Chapter 5]{Caff1} that $u_{\infty}\in C_{\rm loc}^{1,\alpha_{0}}(B_{3/4})$ for some $\alpha_{0}\in(0,1)$ and that $\|u_{\infty}\|_{C^{1,\alpha_{0}}\left(B_{1/2}\right)}\leq C(d,\lambda,\Lambda)$. Finally, taking $h=u_{\infty}$, we reach a contradiction with \eqref{daomaodun} for $j$ sufficiently large and
complete the proof. 
\end{proof}
\subsection{Proof of Theorem \ref{main}}\label{sec3.3}
In this subsection, we present the proof of Theorem \ref{main}. Note that obtaining $C^{1}$ regularity of solutions is showing that the graph of function $u$ can be approximated by an affine function with an error bounded by $C\rho$ in balls of radius $\rho$. To this end, we prove the existence of a sequence of approximating hyperplanes.
  
To begin with, let us consider two moduli of continuity
\begin{equation*}
	\phi(t):=t\sigma_{2}(t)\quad {\rm and}\quad \gamma(t):=\phi^{-1}(t).
\end{equation*}
In what follows, let $h\in C^{1,\alpha_{0}}_{\rm loc}(B_{3/4})$ be a $F$-harmonic function coming from Lemma \ref{lem4.1} and $L>1$ be a constant satisfying
$\|h\|_{C^{1,\alpha_{0}}(B_{1/2})}\leq L$.
Next, we examine the choice of $0<\mu_{1}<1$ by $\gamma(t)$ as follows.

Suppose first $t^{\alpha_{0}}=o(\gamma(t))$, we choose $ 0 < r < \frac{1}{2} $ so small that
\begin{equation}\label{choice1}
	2Lr^{\alpha_{0}}=\gamma(r)=:\mu_{1}>r.
\end{equation}
On the contrary, suppose $O(t^{\alpha_{0}})=\gamma(t)$, we fix $0<\alpha<\alpha_{0}$ and select $0<r<\frac{1}{2}$ so small such that
\begin{equation}\label{choice2}
	2Lr^{\alpha_{0}} = r^{\alpha}=:\mu_{1}>r.
\end{equation}
Observe that once we determine $0<\alpha<\alpha_{0}$, the previous choice becomes universal.

We proceed by defining $0<\vartheta:= \frac{r}{\mu_{1}}<1$ and considering the sequence $\left\{\sigma_{2}^{-1}(\vartheta^{k})\right\}_{k\in\mathbb{N}}$. Since the inverse $ \sigma_{2}^{-1} $ is Dini continuous, we obtain 
that $\left\{\sigma_{2}^{-1}(\vartheta^{k})\right\}_{k\in\mathbb{N}} \in \ell^{1}$. With the choice of $0<\sigma<\frac{1}{4}$ and $\epsilon=\frac{1}{1+\sigma}$, an application of Lemma \ref{lem2.6} yields that there exists a positive sequence $ \{c_{k}\}_{k\in \mathbb{N}}$ that converges to zero and satisfies 
\begin{equation}\label{shoulianguji}
	\frac{7}{10} \sum_{k=1}^{\infty} \sigma_{2}^{-1}(\vartheta^{k}) \leq \sum_{k=1}^{\infty} \frac{\sigma_{2}^{-1}(\vartheta^{k})}{c_{k}}  \leq \sum_{k=1}^{\infty} \sigma_{2}^{-1}(\vartheta^{k})<\infty.
\end{equation}

At this point, we construct a sequence of moduli of continuity $\left\{\Phi^{k}(\cdot,x)\right\}_{k\geq 0}$ given by
\begin{equation}\label{sequence}
	\left\{
	\begin{aligned}
		&  \Phi^{0}(t,x) := \sigma_{1}^{0}(t) + a(x) \sigma_{2}^{0}(t):=\sigma_{1}(t) + a(x) \sigma_{2}(t) ,        \\
		&   \Phi^{1}(t,x) := \sigma_{1}^{1}(t) + a(rx) \sigma_{2}^{1}(t):=\frac{\mu_{1}}{r} \bigg [ \sigma_{1}(\mu_{1}t) + a(rx) \sigma_{2}(\mu_{1}t) \bigg ] ,          \\
		& \Phi^{2}(t,x)  := \sigma_{1}^{2}(t) + a(r^{2}x) \sigma_{2}^{2}(t):=  \frac{\mu_{1}\mu_{2}}{r^{2}} \bigg [ \sigma_{1}(\mu_{1}\mu_{2}t) + a(r^{2}x) \sigma_{2}(\mu_{1}\mu_{2}t) \bigg ]           ,  \\
		 &\qquad\qquad\vdots \\
		&  \Phi^{k}(t,x)  := \sigma_{1}^{k}(t) + a(r^{k}x) \sigma_{2}^{k}(t):= \frac{\mu_{1}\mu_{2}\cdots \mu_{k}}{r^{k}} \bigg [ \sigma_{1}(\mu_{1}\mu_{2}\cdots \mu_{k}t) + a(r^{k}x) \sigma_{2}(\mu_{1}\mu_{2}\cdots \mu_{k}t) \bigg ],
	\end{aligned}
	\right.
\end{equation}
where $\mu_{1}>r$ has been defined as above, and $\left\{\mu_{k}\right\}_{k\geq 2}$ is determined by the following algorithm. If
\begin{equation*}
	\frac{\mu_{k-1}\prod_{i=1}^{k-1}\mu_{i}}{r^{k}}\sigma_{2}\left(\mu_{k-1} \bigg(\prod_{i=1}^{k-1}\mu_{i}\bigg) c_{k}\right) \geq 1,
\end{equation*}
then we set $\mu_{k} = \mu_{k-1}$ and obtain $\sigma_{2}^{k}(c_{k})\geq1$. Otherwise, we choose $\mu_{k}\in(\mu_{k-1},1)$ so that
\begin{equation*}
	\frac{\prod_{i=1}^{k}\mu_{i}}{r^{k}}\sigma_{2} \left(\bigg(\prod_{i=1}^{k}\mu_{i}\bigg) c_{k}\right)  = 1\Longleftrightarrow \sigma_{2}^{k}(c_{k})=1,
\end{equation*}
where $c_{k}$ is the $k$-th element of $\left\{c_{k}\right\}_{k\in\mathbb{N}}$ for which \eqref{shoulianguji} is true.
According to Definition \ref{def2.5}, we know that the sequence of moduli of continuity $\left\{\Phi^{k}(t,x)\right\}_{k\geq 0}$ is shored-up. To proceed, we denote
$$\Gamma:=\big\{\Phi^{0}(t,x), \Phi^{1}(t,x), \cdots, \Phi^{k}(t,x), \cdots\big\}.$$
Thus, with the aid of Lemma \ref{lem2.7}, we conclude that the collection $\Gamma$ is non-collapsing. 

With these ingredients available, we combine them with Lemma \ref{lem4.1} to produce a sequence of affine functions whose difference with respect to $u$ grows in a controlled fashion.
\begin{lemma}\label{4lem1}
	Suppose that the hypotheses \eqref{A1}-\eqref{A6} and \eqref{phiform1} hold. Let $u\in C(B_{1})$ be a viscosity solution of \eqref{4raodongfangcheng} with $\|u\|_{L^{\infty}(B_{1})}\leq 1$.	
	There exists a constant $\delta>0$ coming from Lemma \ref{lem4.1} such that if
	\begin{equation*}
		\max\left\{\|{\rm osc}_{{F}}\|_{L^{\infty}\left(B_{1}\right)},\|f\|_{L^{\infty}(B_{1})},\mathcal{M}\right\}\leq \delta,
	\end{equation*}
	then there exist an affine function 
	$l(x)=\mathcal{A}_{1}+\mathcal{B}_{1}\cdot x$ such that
	\begin{equation}\label{4chushidiedai}
		|\mathcal{A}_{1}|+|\mathcal{B}_{1}|\leq C \quad {\rm and} \quad \sup_{x \in B_{r}} |u(x)-l(x)|\leq \mu_{1}r.
	\end{equation}
where universal constant $0<r<\frac{1}{2}$ is given in \eqref{choice1} and \eqref{choice2}, and $C>0$ is a universal constant.
\end{lemma}
\begin{proof}
	By means of Lemma \ref{lem4.1}, we know that there exists a $F$-harmonic function $h\in C^{1,\alpha_{0}}_{\rm loc}(B_{3/4})$ such that
	\begin{equation}\label{sec3:eq38}
		\sup_{x\in B_{1/2}} |u(x)-h(x)| \leq \varepsilon
	\end{equation}
	for some $\varepsilon>0$ to be determined later. Once $\varepsilon$ is chosen, we may determine the value of $\delta$ in Lemma \ref{lem4.1}. The regularity of $h$ implies that for every $0<r<\frac{1}{2}$, there exists a constant $L>1$ such that
	\begin{equation}\label{sec3:eq39}
		\sup_{x \in B_{r}} |h(x)-h(0)-Dh(0)\cdot x| \leq Lr^{1+\alpha_{0}}
	\end{equation}
	with $|h(0)|+|Dh(0)|\leq L$. Letting $\mathcal{A}_{1}:=h(0)$ and $\mathcal{B}_{1}:=Dh(0)$, combining (\ref{sec3:eq38}) with (\ref{sec3:eq39}) and the choice of $\mu_{1}$, we arrive at 
	\begin{align*}
		\sup_{x \in B_{r}} |u(x)-\mathcal{A}_{1}-\mathcal{B}_{1}\cdot x|  & \leq \sup_{x \in B_{r}} |u(x)-h(x)| + \sup_{x \in B_{r}} |h(x)-\mathcal{A}_{1}-\mathcal{B}_{1}\cdot x|  \\
		& \leq \varepsilon + Lr^{1+\alpha_{0}} = \varepsilon + \frac{1}{2}\mu_{1} r.
	\end{align*}
 Now choosing $\varepsilon=\frac{1}{2}\mu_{1}r$, there holds
	\begin{equation*}
		\sup_{x \in B_{r}} |u(x)-\mathcal{A}_{1}-\mathcal{B}_{1}\cdot x| \leq \mu_{1} r.
	\end{equation*}
	Hence the proof is completed.
\end{proof}
Now we extrapolate the findings in Lemma \ref{4lem1} to arbitrary small scales, in a discrete scheme.
\begin{lemma}\label{4lem2}
	Suppose that the hypotheses \eqref{A1}-\eqref{A6} and \eqref{phiform1} are in force. Let $ u \in C({B}_{1}) $ be a viscosity solution to \eqref{4doublephasezhufangcheng} with $\|u\|_{L^{\infty}(B_{1})}\leq 1$. Assume that
	\begin{equation}\label{4smallcondition}
		\max\left\{\|{\rm osc}_{{F}}\|_{L^{\infty}\left(B_{1}\right)},\|f\|_{L^{\infty}(B_{1})},\mathcal{M}\right\}\leq \delta
	\end{equation}
with $\delta>0$ given in Lemma \ref{4lem1}. Then there exists a sequence of affine functions $ \left\{l_{k}\right\}_{k\in \mathbb{N}}  $ of the form $l_{k}(x)= \mathcal{A}_{k} + \mathcal{B}_{k} \cdot x$ satisfying
	\begin{equation}\label{4diedaijielun}
		\sup_{x \in B_{r^{k}}} |u(x)-l_{k}(x)| \leq \bigg(\prod_{i=1}^{k}\mu_{i} \bigg) r^{k},
	\end{equation}
	\begin{equation}\label{cauchylie}
		|\mathcal{A}_{k+1}-\mathcal{A}_{k}| \leq C \bigg(\prod_{i=1}^{k}\mu_{i} \bigg) r^{k}\quad {\rm and}\quad |\mathcal{B}_{k+1}-\mathcal{B}_{k}| \leq C \prod_{i=1}^{k}\mu_{i}
	\end{equation}
	for every $k\in \mathbb{N}$, where $0<r<\frac{1}{2}$ is given in Lemma \ref{4lem1} and $C>0$ is a universal constant.
\end{lemma}
\begin{proof}
We resort to an induction argument. Notice that \eqref{4diedaijielun} holds for $k=1$ with $l_{1}(x):=l(x)$ by Lemma \ref{4lem1}. Set $\overline{l}_{0}(x):=\overline{\mathcal{A}}_{0}+\overline{\mathcal{B}}_{0}\cdot x=l_{1}(x)$. To begin with, let us define the auxiliary function $u_{1}:B_{1}\rightarrow \mathbb{R}$ by
	\begin{equation*}
		u_{1}(x):= \frac{u(rx)-\overline{l}_{0}(rx)}{\mu_{1}r}
	\end{equation*}
	with $\mu_{1}$ as in Lemma \ref{4lem1}. It is easy to verify that $u_{1}$ is a viscosity solution to
	\begin{equation*}
	\Phi^{1}\bigg(\Abs{Du_{1}+\frac{\overline{\mathcal{B}}_{0}}{\mu_{1}}},x\bigg)F_{1}(D^2 u_{1}, x)+H_{1}\left(Du_{1}+\frac{\overline{\mathcal{B}}_{0}}{\mu_{1}},x\right)=f_{1}(x)\quad  \text{in} \quad B_{1},
	\end{equation*}
	where $\Phi^{1}$ is defined in \eqref{sequence},
	\begin{align*}
		 F_{1}(X,x):=\frac{r}{\mu_{1}}F\left(\frac{\mu_{1}}{r}X,rx\right),\quad H_{1}(t,x):=H(\mu_{1}t,rx),\quad f_{1}(x):=f(rx).
	\end{align*}
	Note that $F_{1}$ is still a uniformly $(\lambda,\Lambda)$-elliptic operator and
	\begin{equation*}
		{\rm osc}_{{F}_{1}}(x,0)=\sup\limits_{M\in S^{d}\setminus \{0\}}\frac{\Abs{F\left(\frac{\mu_{1}}{r}M,rx\right)-F\left(\frac{\mu_{1}}{r}M,0\right)}}{\frac{\mu_{1}}{r}\|M\|}={\rm osc}_{{F}}(r x,0).
	\end{equation*}
It follows from \eqref{4smallcondition} and $r<\frac{1}{2}$ that
	\begin{equation*}
	\max\left\{\|{\rm osc}_{{F_{1}}}\|_{L^{\infty}\left(B_{1}\right)},\|f_{1}\|_{L^{\infty}(B_{1})}\right\}\leq \delta.
	\end{equation*} 
Also, it follows from \eqref{4diedaijielun} for $k=1$ that $\|u_{1}\|_{L^{\infty}\left(B_{1}\right)}\leq 1$. By the argument before Lemma \ref{4lem1}, we have $\sigma_{2}^{1}(1)\geq 1$. By virtue of \eqref{A5} and $\frac{r}{\mu_{1}}<1$, we arrive at
	\begin{equation*}
		 \abs{H_{1}(t,x)}\leq \mathcal{M}\left(1+\sigma_{1}(\mu_{1}|t|)\right)\leq \mathcal{M}\left(1+\frac{\mu_{1}}{r}\sigma_{1}(\mu_{1}|t|)\right)=\mathcal{M}\left(1+\sigma_{1}^{1}(|t|)\right).
	\end{equation*}
At this point, we have verified that $u_{1}$ falls into the framework of Lemma \ref{4lem1}. Hence, there exists an affine function $\overline{l}_{1}(x)=\overline{\mathcal{A}}_{1}+\overline{\mathcal{B}}_{1}\cdot x$ with $|\overline{\mathcal{A}}_{1}|+|\overline{\mathcal{B}}_{1}|\leq C$ such that
	\begin{equation*}
		\sup_{x \in B_{r}} |u_{1}(x)-\overline{l}_{1}(x)| \leq \mu_{1}r.
	\end{equation*}
	
	Next, we define $u_{2}:B_{1}\rightarrow \mathbb{R}$ as
	\begin{equation*}
		u_{2}(x):= \frac{u_{1}(rx)-\overline{l}_{1}(rx)}{\mu_{2}r},
	\end{equation*}
	for $r<\mu_{1}\leq \mu_{2}$ chosen earlier. Then $u_{2}$ is a viscosity solution to
	\begin{equation*}
	\Phi^{2}\bigg(\Abs{Du_{2}+\frac{\overline{\mathcal{B}}_{1}}{\mu_{2}}+\frac{\overline{\mathcal{B}}_{0}}{\mu_{1}\mu_{2}}},x\bigg)F_{2}(D^2 u_{2}, x)+H_{2}\left(Du_{2}+\frac{\overline{\mathcal{B}}_{1}}{\mu_{2}}+\frac{\overline{\mathcal{B}}_{0}}{\mu_{1}\mu_{2}},x\right)=f_{2}(x)\quad  \text{in} \quad B_{1},
\end{equation*}
	where $\Phi^{2}$ is given in \eqref{sequence},
		\begin{align*}
		 &F_{2}(X,x):=\frac{r}{\mu_{2}}F_{1}\left(\frac{\mu_{2}}{r}X,rx\right)=\frac{r^{2}}{\mu_{1}\mu_{2}}F\left(\frac{\mu_{1}\mu_{2}}{r^{2}}X,r^{2}x\right),\\
		&H_{2}(t,x):=H_{1}(\mu_{2}t,rx)=H(\mu_{1}\mu_{2}t,r^{2}x),\quad f_{2}(x):=f(r^{2}x).
	\end{align*}
	Note that $F_{2}$ is still a uniformly $(\lambda,\Lambda)$-elliptic operator and
	\begin{equation*}
		\|u_{2}\|_{L^{\infty}\left(B_{1}\right)}\leq 1,\quad 
		\max\left\{\|{\rm osc}_{{F_{2}}}\|_{L^{\infty}\left(B_{1}\right)},\|f_{2}\|_{L^{\infty}(B_{1})}\right\}\leq \delta.
	\end{equation*} 
 In view of \eqref{A5} and $r<\mu_{1}\leq \mu_{2}$, we have
	\begin{equation*}
		\abs{H_{2}(t,x)}\leq \mathcal{M}\left(1+\sigma_{1}(\mu_{1}\mu_{2}|t|)\right)\leq \mathcal{M}\left(1+\frac{\mu_{1}\mu_{2}}{r^{2}}\sigma_{1}(\mu_{1}\mu_{2}|t|)\right)=\mathcal{M}\left(1+\sigma_{1}^{2}(|t|)\right).
	\end{equation*}
	At this moment, $u_{2}$ falls into the framework of Lemma \ref{4lem1}, and hence there exists an affine function $\overline{l}_{2}(x)=\overline{\mathcal{A}}_{2}+\overline{\mathcal{B}}_{2}\cdot x$ with $|\overline{\mathcal{A}}_{2}|+|\overline{\mathcal{B}}_{2}|\leq C$ such that
	\begin{equation*}
		\sup_{x \in B_{r}} |u_{2}(x)-\overline{l}_{2}(x)| \leq \mu_{1}r.
	\end{equation*}
	
	Recursively, let us define
		\begin{equation*}
		u_{k}(x):= \frac{u_{k-1}(rx)-\overline{l}_{k-1}(rx)}{\mu_{k}r},
	\end{equation*}
	for $r<\mu_{1}\leq \mu_{2}\leq \cdots\leq \mu_{k-1}\leq \mu_{k}$ chosen earlier. Then $u_{k}$ is a viscosity solution to
	\begin{equation*}
		\Phi^{k}\bigg(\Abs{Du_{k}+\sum_{i=1}^{k}\frac{\overline{\mathcal{B}}_{i-1}}{\mu_{i}\mu_{i+1}\cdots\mu_{k}}},x\bigg)F_{k}(D^2 u_{k}, x)+H_{k}\left(Du_{k}+\sum_{i=1}^{k}\frac{\overline{\mathcal{B}}_{i-1}}{\mu_{i}\mu_{i+1}\cdots\mu_{k}},x\right)=f_{k}(x)\;\; \text{in}\; B_{1},
	\end{equation*}
	where $\Phi^{k}$ is given in \eqref{sequence},
	\begin{align*}
	& F_{k}(X,x):=\frac{r^{k}}{\mu_{1}\mu_{2}\cdots \mu_{k}}F\left(\frac{\mu_{1}\mu_{2}\cdots \mu_{k}}{r^{k}}X,r^{k}x\right),\\
	&H_{k}(t,x):=H_{k-1}(\mu_{k}t,rx)=\cdots=H\left(\bigg( \prod_{i=1}^{k} \mu_{i}\bigg)t,r^{k}x\right),\quad f_{k}(x):=f(r^{k}x).
\end{align*}
	Note that $F_{k}$ is still a uniformly $(\lambda,\Lambda)$-elliptic operator and
	\begin{equation*}
		\|u_{k}\|_{L^{\infty}\left(B_{1}\right)}\leq 1,\quad 
		\max\left\{\|{\rm osc}_{{F_{k}}}\|_{L^{\infty}\left(B_{1}\right)},\|f_{k}\|_{L^{\infty}(B_{1})}\right\}\leq \delta.
	\end{equation*} 
	By means of \eqref{A5} and  $r^{-k}\prod_{i=1}^{k} \mu_{i}\geq 1$, we have
	\begin{equation*}
		\abs{H_{k}(t,x)}\leq \mathcal{M}\left(1+\sigma_{1}\left(|t| \prod_{i=1}^{k} \mu_{i}\right)\right)\leq \mathcal{M}\left(1+\left(r^{-k}\prod_{i=1}^{k} \mu_{i}\right)\sigma_{1}\bigg( |t|\prod_{i=1}^{k} \mu_{i}\bigg)\right)=\mathcal{M}\left(1+\sigma_{1}^{k}(|t|)\right).
	\end{equation*}
	At this stage, $u_{k}$ falls into the framework of Lemma \ref{4lem1}, and thus there exists an affine function $\overline{l}_{k}(x)=\overline{\mathcal{A}}_{k}+\overline{\mathcal{B}}_{k}\cot x$ with $|\overline{\mathcal{A}}_{k}|+|\overline{\mathcal{B}}_{k}|\leq C$ such that
	\begin{equation*}
		\sup_{x \in B_{r}} |u_{k}(x)-\overline{l}_{k}(x)| \leq \mu_{1}r.
	\end{equation*}
	Scaling back to $ u $ yields that
	\begin{equation*}
		\sup_{x \in B_{r^{k+1}}} |u(x)-{l}_{k+1}(x)| \leq \mu_{1}^{2} \mu_{2} \cdots \mu_{k} r^{k+1} \leq \bigg( \prod_{i=1}^{k+1} \mu_{i}   \bigg)r^{k+1},
	\end{equation*}
	where 
	\begin{equation*}
		l_{k+1}(x) := \overline{l}_{0}(x) + \sum_{i=1}^{k}\overline{l}_{i}(r^{-i}x)\bigg( \prod_{j=1}^{i} \mu_{j}\bigg)r^{i}:=\mathcal{A}_{k+1}+\mathcal{B}_{k+1}\cdot x
	\end{equation*}
with
	\begin{equation*}
		\mathcal{A}_{k+1}:= \overline{\mathcal{A}}_{0}+\sum_{i=1}^{k}\bigg(\prod_{j=1}^{i} \mu_{j}\bigg)r^{i}\overline{\mathcal{A}}_{i},\quad \mathcal{B}_{k+1}:= \overline{\mathcal{B}}_{0}+\sum_{i=1}^{k}\bigg(\prod_{j=1}^{i} \mu_{j}\bigg)\overline{\mathcal{B}}_{i}.
	\end{equation*}
	As a consequence, for every $k\in\mathbb{N}$, we get
	\begin{equation*}
		|\mathcal{A}_{k+1}-\mathcal{A}_{k}| = \bigg| \bigg(\prod_{i=1}^{k} \mu_{i}\bigg)r^{k}\overline{\mathcal{A}}_{k}\bigg| \leq C \bigg(\prod_{i=1}^{k} \mu_{i} \bigg)r^{k}
	\end{equation*}
	and
	\begin{equation*}
		|\mathcal{B}_{k+1}-\mathcal{B}_{k}| = \bigg|\bigg( \prod_{i=1}^{k} \mu_{i}\bigg)\overline{\mathcal{B}}_{k}   \bigg| \leq C \bigg(\prod_{i=1}^{k} \mu_{i}    \bigg).
	\end{equation*}
	This completes the proof.
\end{proof}

In the end, with the help of Lemma \ref{4lem2}, we are ready to present the proof of Theorem \ref{main}.
\begin{proof}[Proof of Theorem \ref{main}] To prove Theorem \ref{main}, we first reduce the problem to a smallness regime in Lemma \ref{4lem2}, namely, 
\begin{equation*}
	\|{u}\|_{L^{\infty}\left(B_{1}\right)}\leq 1\quad {\rm and }\quad
	\max\left\{\|{\rm osc}_{{F}}\|_{L^{\infty}\left(B_{1}\right)},\|{f}\|_{L^{\infty}\left(B_{1}\right)},{\mathcal{M}}\right\}
	\leq \delta.
\end{equation*}	 In fact, we define $\tilde{u}:B_{1}\rightarrow \mathbb{R}$ by
$$\tilde{u}(x)=\frac{u(\iota x)}{K}$$
with $0<\iota\leq 1\leq K$ to be chosen later. We can readily check that
$\tilde{u}$ is a viscosity solution of
	\begin{equation}\label{model}
		\tilde{\Phi}(|D\tilde{u}|,x)\tilde{F}(D^2 \tilde{u}, x)+\tilde{H}(D\tilde{u},x) =\tilde{f}(x) \quad  \text{in} \quad B_{1},
	\end{equation}
where 
\begin{equation*}
	\begin{split}
		\tilde{\Phi}(t,x):=&\tilde{\sigma}_{1}(t)+\tilde{a}(x)\tilde{\sigma}_{2}(t), \quad
		\tilde{\sigma}_{i}(t):=\sigma_{i}\left(\frac{K}{\iota}t\right),\;i=1,2,\quad \tilde{a}(x):=a(\iota x),\\
		\tilde{F}(X,x):=&\frac{\iota^{2}}{K}F\left(\frac{K}{\iota^{2}}X,\iota x\right),\quad
		\tilde{H}(t,x):=\frac{\iota^{2}}{k}H\left(\frac{K}{\iota}t,\iota x\right),\quad
		\tilde{f}(x):=\frac{\iota^{2}}{K}f(\iota x).
	\end{split}
\end{equation*}
Note that $\tilde{F}$ is still a uniformly $(\lambda,\Lambda)$-elliptic operator and
\begin{equation*}
	{\rm osc}_{\tilde{F}}(x,0)=\sup\limits_{M\in S^{d}\setminus \{0\}}\frac{\Abs{F\left(\frac{K}{\iota^{2}}M,\iota x\right)-F\left(\frac{K}{\iota^{2}}M,0\right)}}{\frac{K}{\iota^{2}}\|M\|}={\rm osc}_{{F}}(\iota x,0).
\end{equation*}
This along with \eqref{A2} immediately leads to
\begin{equation*}
	\|{\rm osc}_{\tilde{F}}\|_{L^{\infty}\left(B_{1}\right)}= \|{\rm osc}_{{F}}\|_{L^{\infty}\left(B_{\iota}\right)}\leq C_{F}\iota^{\theta}.
\end{equation*}
In addition, observe that
$$\tilde{\sigma}_{i}^{-1}(t):=\frac{\iota}{K}\sigma_{i}^{-1}(t),\;\;i=1,2.$$
In fact, $$\tilde{\sigma}_{i}^{-1}(\tilde{\sigma}_{i}(t))=\tilde{\sigma}_{i}^{-1}\left(\sigma_{i}\left(\frac{K}{\iota}t\right)\right)=\frac{\iota}{K}\sigma_{i}^{-1}\left(\sigma_{i}\left(\frac{K}{\iota}t\right)\right)=t,\quad i=1,2.$$
It follows from \eqref{A3}, \eqref{A4} and $\iota\leq K$ that
\begin{equation*}
	\int_{0}^{1}\frac{\tilde{\sigma}_{i}^{-1}(t)}{t}\text{d}t\leq \int_{0}^{1}\frac{\sigma_{i}^{-1}(t)}{t}\text{d}t<\infty, \quad i=1,2,
\end{equation*}
and $0\leq \tilde{a}(\cdot)\in C(B_{1})$. Also, according to \eqref{A5}, we have 
\begin{equation*}
	\Abs{\tilde{H}(t,x)}\leq \frac{\iota^{2}}{K}\mathcal{M}\left(1+\sigma_{1}\left(\frac{K}{\iota}|t|\right)\right)=:\tilde{\mathcal{M}}\left(1+\tilde{\sigma}_{1}(|t|)\right). 
\end{equation*}
Now, for given $\delta\in(0,1)$, we select
\begin{equation*}
	K:=1+\|u\|_{L^{\infty}\left(B_{1}\right)}+\|f\|_{L^{\infty}\left(B_{1}\right)}+\mathcal{M},\quad
\iota:=\min\left\{1,\sqrt{\delta},\left(\frac{\delta}{C_{F}}\right)^{1/\theta}\right\}.
\end{equation*}
With such choice, we arrive at
\begin{equation*}
	\|\tilde{u}\|_{L^{\infty}\left(B_{1}\right)}\leq 1\quad {\rm and }\quad
	\max\left\{\|{\rm osc}_{\tilde{F}}\|_{L^{\infty}\left(B_{1}\right)},\|\tilde{f}\|_{L^{\infty}\left(B_{1}\right)},\tilde{\mathcal{M}}\right\}
	\leq \delta.
\end{equation*}	
Therefore, $\tilde{u}$ solves an equation possessing the same structure as \eqref{4doublephasezhufangcheng} with $\Phi$ of the form \eqref{phiform1} and $\tilde{u}$ is in the smallness regime.

Next, we examine the convergence of the sequence of $\left\{\prod_{i=1}^{k} \mu_{i}\right\}_{k\in \mathbb{N}}$ given in Lemma \ref{4lem2}. Notice that two possibilities concerning the sequence $\left\{\mu_{k}\right\}_{k\in\mathbb{N}} $ will happen. Either the sequence repeats after some index $ N \geq 2 $ or we have $ \mu_{k} < \mu_{k+1} $ for infinitely many indices $k\in \mathbb{N} $.
	
	In the case of $\mu_{N}=\mu_{N+1}=\mu_{N+2}=\cdots$, since $r<\mu_{1}=2Lr^{\alpha_{0}}\leq \mu_{2}\leq \cdots\leq \mu_{N-1} <\mu_{N}<1$, we can choose $\beta\in (0,\alpha_{0})$ such that $r^{\beta}=\mu_{N}$. Then \eqref{4diedaijielun} and \eqref{cauchylie} immediately become
		\begin{equation*}
		\sup_{x \in B_{r^{k}}} |u(x)-l_{k}(x)| \leq \mu_{N}^{k}r^{k}=r^{k(1+\beta)},
	\end{equation*}
	\begin{equation*}
		|\mathcal{A}_{k+1}-\mathcal{A}_{k}| \leq C  r^{k(1+\beta)}\quad {\rm and}\quad |\mathcal{B}_{k+1}-\mathcal{B}_{k}| \leq C r^{k\beta}.
	\end{equation*}
	By following the standard argument (see \cite{Fang,Huo2026}), we can conclude that solution $u\in C^{1,\beta}_{\rm loc}(B_{1})$ for some $\beta\in (0,\alpha_{0})$, where $\alpha_{0}\in (0,1)$ is the exponent associated with the regularity of $F$-harmonic function. 
	
	In the latter case, there holds
	\begin{equation*}
		\frac{\prod_{i=1}^{k+1}\mu_{i}}{r^{k+1}}\sigma_{2} \bigg(\prod_{i=1}^{k+1}\mu_{i} c_{k+1} \bigg) = 1 \Longleftrightarrow \sigma_{2}^{k+1}(c_{k+1})=1.
	\end{equation*}
	This implies that
	\begin{equation}\label{4shouliantiaojian1}
		 \prod_{i=1}^{k+1} \mu_{i}  =  \frac{1}{c_{k+1}} \sigma_{2}^{-1} \bigg(\frac{r^{k+1}}{\prod_{i=1}^{k+1}\mu_{i}}   \bigg) \leq \frac{\sigma_{2}^{-1}(\vartheta^{k+1})}{c_{k+1}},
	\end{equation}
where in the last inequality, we used that $\sigma_{2}$ is an increasing function,
$\vartheta=\frac{r}{\mu_{1}}$, and $\prod_{i=1}^{k+1}\mu_{i}\geq \mu_{1}^{k+1}$.
Now we define $\left\{\gamma_{k}\right\}_{k\in \mathbb{N}}$ as
$$\gamma_{k}:=\prod_{i=1}^{k} \mu_{i}.$$
	A combination of \eqref{4shouliantiaojian1} and \eqref{shoulianguji} shows that $\left\{\gamma_{k}\right\}_{k\in \mathbb{N}} \in \ell^{1}$ and
	$$\sum_{k=1}^{\infty}\gamma_{k}\leq\sum_{k=1}^{\infty}\frac{\sigma_{2}^{-1}(\vartheta^{k})}{c_{k}}
	\leq \sum_{k=1}^{\infty} \sigma_{2}^{-1}(\vartheta^{k})<\infty.$$
	Consequently, we deduce that
	\begin{equation*}
		\lim_{k\rightarrow \infty} \bigg( \prod_{i=1}^{k} \mu_{i} \bigg) =0.
	\end{equation*}
	Thereby, it follows from \eqref{cauchylie} that
	$\left\{\mathcal{A}_{k}\right\}_{k\in \mathbb{N}}$ and 
	$\left\{\mathcal{B}_{k}\right\}_{k\in \mathbb{N}}$ are Cauchy sequences. Then there exist $ \mathcal{A}_{\infty} \in \mathbb{R} $ and $ \mathcal{B}_{\infty} \in \mathbb{R}^{d} $ such that
	\begin{equation*}
		\mathcal{A}_{k} \rightarrow \mathcal{A}_{\infty} \ \ \text{and} \ \ \mathcal{B}_{k} \rightarrow \mathcal{B}_{\infty}
	\end{equation*}
	as $ k \rightarrow \infty $. Moreover, for any $n> k$, a combination of the triangle inequality and \eqref{cauchylie} yields that 
	$$\abs{\mathcal{A}_{k}-\mathcal{A}_{n}}\leq \sum_{j=k}^{n-1}\abs{\mathcal{A}_{j}-\mathcal{A}_{j+1}}\leq C\sum_{j=k}^{n-1} \bigg(\prod_{i=1}^{j} \mu_{i} \bigg)r^{j}\leq C\sum_{j=k}^{n-1} \gamma_{j}r^{k},
	$$
	$$\abs{\mathcal{B}_{k}-\mathcal{B}_{n}}\leq \sum_{j=k}^{n-1}\abs{\mathcal{B}_{j}-\mathcal{B}_{j+1}}\leq C\sum_{j=k}^{n-1} \bigg(\prod_{i=1}^{j} \mu_{i} \bigg)= C\sum_{j=k}^{n-1} \gamma_{j}.
	$$
	Letting $n \rightarrow \infty$, we get
	\begin{equation}\label{Sec3:eq38}
		|\mathcal{A}_{k}-\mathcal{A}_{\infty}| \leq C \bigg( \sum_{j=k}^{\infty} \gamma_{j}\bigg) r^{k} \ \ \text{and} \ \  |\mathcal{B}_{k}-\mathcal{B}_{\infty}| \leq C \bigg( \sum_{j=k}^{\infty} \gamma_{j}  \bigg).
	\end{equation}
	Set $l_{\infty}(x):= \mathcal{A}_{\infty} + \mathcal{B}_{\infty}\cdot x $. For any $0< \rho \leq r$, then there exists $k \in \mathbb{N}$ such that $r^{k+1} < \rho \leq r^{k}$.
	Combining Lemma \ref{4lem2} with \eqref{Sec3:eq38} and the triangle inequality, we arrive at
	\begin{align}\label{Sec3:eq39}
		\begin{split}
			\sup_{x\in B_{\rho}} |u(x)-l_{\infty}(x)| & \leq \sup_{x\in B_{r^{k}}} |u(x)-l_{k}(x)|  + \sup_{x\in B_{r^{k}}}  |l_{k}(x)-l_{\infty}(x)|   \\
			& \leq C \gamma_{k} r^{k}  +  C \bigg( \sum_{i=k}^{\infty} \gamma_{i}\bigg) r^{k}  \\
			& \leq C\bigg( \sum_{i=k}^{\infty} \gamma_{i}\bigg) \rho .
		\end{split}
	\end{align}
	In the end, let us set
	\begin{equation*}
		\omega(t) := \sum_{i=\lfloor\ln t^{-1}\rfloor}^{\infty} \gamma_{i},
	\end{equation*}
	where $ \lfloor\ln t^{-1}\rfloor $ denotes the biggest integer that is less than or equal to $ \ln t^{-1} $. Since $ \left\{\gamma_{i}\right\} \in \ell^{1} $, then $\omega(t)$ indeed is a modulus of continuity. Hence \eqref{Sec3:eq39} becomes
	\begin{equation*}
		\sup_{x\in B_{\rho}} |u(x)-l_{\infty}(x)|  \leq C \omega(\rho) \rho,
	\end{equation*}
	which finishes the proof of Theorem \ref{main}.
\end{proof}

\section{$C^{1}$ regularity to \eqref{4doublephasezhufangcheng} with $\Phi$ of the form \eqref{xintiaojian}}
In this section, we complete the proof of Theorems \ref{main22} and \ref{main333}, which relate
to the $C^{1}$ and $C^{1,\alpha}$ regularity estimates, respectively. We start off by proving local Lipschitz regularity for a viscosity solution of \eqref{4doublephasezhufangcheng} with $\Phi$ of the form \eqref{xintiaojian}.
	\begin{proposition}\label{qiyituihualipjinxing}
	Assume that \eqref{A1}-\eqref{A6} and \eqref{xintiaojian} hold.
	Let $u\in C(B_{1})$ be a viscosity solution to \eqref{4doublephasezhufangcheng} with $\|u\|_{L^{\infty}(B_{1})}\leq 1$.
	Then $u$ is locally Lipschitz
	continuous in $B_{1}$ with the estimate  
	$$[u]_{C^{0,1}(B_{1/2})}\leq C(d,\lambda,\Lambda,C_{F},\theta,\sigma_{1}(1),\mathcal{M},\|f\|_{L^{\infty}(B_{1})}).$$
\end{proposition}
\begin{proof}
	The proof follows along similar lines as the proof of Proposition \ref{4aajinxingjieguo2}-(i). We only highlight the key differences. Here, we have the following viscosity inequalities
	\begin{equation}\label{jie1}
		\Phi(\abs{\xi_{\hat{x}}},\hat{x})F(X, \hat{x})+
		H({\xi_{\hat{x}}},\hat{x})\geq f(\hat{x}),
	\end{equation}
	\begin{equation}\label{jie2}
		\Phi(\abs{\xi_{\hat{y}}},\hat{y})F(Y, \hat{y})+H({\xi_{\hat{y}}},\hat{y})\leq f(\hat{y}),
	\end{equation}
	where $\xi_{\hat{x}},\xi_{\hat{y}}$ are given in \eqref{4aasubsuperjet}.  
	Let us recall that
	\begin{equation*}
		F(X,\hat{x})\leq  P_{\lambda,\Lambda}^{+}(X-Y)+F(Y,\hat{y})+C_{F}\|Y\|\abs{\hat{x}-\hat{y}}^{\theta},\quad  \|Y\|\leq L_{1}\abs{\hat{x}-\hat{y}}^{-1}+2L_{2}+\epsilon,
	\end{equation*}
	\begin{equation*}
		P_{\lambda,\Lambda}^{+}(X-Y)\leq \left(\lambda+\Lambda(d-1)\right)(4L_{2}+2\epsilon)-4L_{1}\lambda\beta(\beta+1)w_{0}\abs{\hat{x}-\hat{y}}^{\beta-1}.
	\end{equation*}
	 We choose $L_{1}\geq \max \{4,L_{2}\}$ and $L_{2}\geq \left(2^{2+\beta}w_{0}(1+\beta)\right)^{2/\beta}$ so that
	$\abs{\xi_{\hat{x}}}< 2L_{1}$ and
	\begin{align*}
		\abs{\xi_{\hat{x}}}\geq& L_{1}\omega^{\prime}(\abs{\hat{x}-\hat{y}})-\frac{L_{2}}{2}=L_{1}\left(1-w_{0}(1+\beta)\abs{\hat{x}-\hat{y}}^{\beta}\right)-\frac{L_{2}}{2}\\
		\geq&L_{1}\left(1-w_{0}(1+\beta)\left(\frac{2}{\sqrt{L_{2}}}\right)^{\beta}\right)-\frac{L_{2}}{2}\geq \frac{3L_{1}}{4}-\frac{L_{2}}{2}\geq \frac{L_{1}}{4}\geq 1,
	\end{align*}
where the fact $\abs{\hat{x}-\hat{y}}\leq \frac{2}{\sqrt{L_{2}}}$ is used.
	Also, we have $1\leq \frac{L_{1}}{4}\leq \abs{\xi_{\hat{y}}}\leq 2L_{1}$. 
	Based on the above estimates, together with the assumptions \eqref{A3}-\eqref{A5}, we deduce that
	\begin{equation*}
		\frac{f(\hat{y})-H({\xi_{\hat{y}}},\hat{y})}{\Phi\left(\abs{\xi_{\hat{y}}},\hat{y}\right)}\leq |\xi_{\hat{y}}|  \frac{\|f\|_{L^{\infty}(B_{1})}+\mathcal{M}\bigg(1+\sigma_{1}(\abs{\xi_{\hat{y}}})\bigg)}{\sigma_{1}(\abs{\xi_{\hat{y}}})}\leq 
		2L_{1}\bigg( \mathcal{M}+\frac{\mathcal{M}+\|f\|_{L^{\infty}(B_{1})}}{\sigma_{1}(1)}\bigg),
	\end{equation*}
	\begin{equation*}
		\frac{f(\hat{x})-H({\xi_{\hat{x}}},\hat{x})}{\Phi\left(\abs{\xi_{\hat{x}}},\hat{y}\right)}\geq- |\xi_{\hat{x}}|  \frac{\|f\|_{L^{\infty}(B_{1})}+\mathcal{M}\bigg(1+\sigma_{1}(\abs{\xi_{\hat{x}}})\bigg)}{\sigma_{1}(\abs{\xi_{\hat{x}}})}\geq 
		-2L_{1}\bigg( \mathcal{M}+\frac{\mathcal{M}+\|f\|_{L^{\infty}(B_{1})}}{\sigma_{1}(1)}\bigg).
	\end{equation*}
Combining with the above information, we arrive at
\begin{equation*}
	4L_{1}\lambda\beta(\beta+1)w_{0}\abs{\hat{x}-\hat{y}}^{\beta-1}\leq C_{1}+C_{F}L_{1}\Abs{\hat{x}-\hat{y}}^{\theta-1}+4L_{1}\bigg( \mathcal{M}+\frac{\mathcal{M}+\|f\|_{L^{\infty}(B_{1})}}{\sigma_{1}(1)}\bigg),
\end{equation*}
where $C_{1}:=\left(\Lambda(d-1)+\lambda\right)(4L_{2}+2)+C_{F}(2L_{2}+1)$. This along with $\abs{\hat{x}-\hat{y}}^{\theta-1}>1$ yields that
\begin{equation*}
	0\leq C_{1}+4L_{1}\abs{\hat{x}-\hat{y}}^{\beta-1}\bigg(C_{2}
	\Abs{\hat{x}-\hat{y}}^{\theta-\beta}-\lambda\beta(\beta+1)w_{0}\bigg).
\end{equation*}
where $C_{2}:=\frac{C_{F}}{4}+\mathcal{M}+\frac{\mathcal{M}+\|f\|_{L^{\infty}(B_{1})}}{\sigma_{1}(1)} $.
In light of $\abs{\hat{x}-\hat{y}}\leq \frac{2}{\sqrt{L_{2}}}<1$ and $\beta\in (0,\theta)$, we choose $L_{2}\geq 4\left(\frac{4C_{2}}{\lambda w_{0}\beta(1+\beta)}\right)^{2/(\theta-\beta)}$ so that
\begin{equation*}
	C_{2}\Abs{\hat{x}-\hat{y}}^{\theta-\beta}\leq C_{2}\left(\frac{2}{\sqrt{L_{2}}}\right)^{\theta-\beta}\leq \frac{1}{4}\lambda w_{0}\beta(1+\beta).
\end{equation*}
Then it follows that
\begin{equation}\label{qiyimaodun222}
	3L_{1}\lambda w_{0}\beta(1+\beta)< 3L_{1}\lambda w_{0}\beta(1+\beta)\Abs{\hat{x}-\hat{y}}^{\beta-1}\leq C_{1}.
\end{equation}
Finally, taking $L_{1}>\frac{C_{1}}{3\lambda w_{0}\beta(1+\beta)}$, we obtain a contradiction with \eqref{qiyimaodun222}. This completes the proof.
\end{proof}

Building upon the local Lipschitz regularity of solutions above and Theorem \ref{main}, we are now in a position to complete the proof of Theorem \ref{main22}.
\begin{proof}[Proof of Theorem \ref{main22}]
First, we can assume that $\|u\|_{L^{\infty}(B_{1})}\leq 1$ by normalization. Indeed, we take $K:=1+\|u\|_{L^{\infty}\left(B_{1}\right)}$ and consider scaled function $\tilde{u}(x):=\frac{u(x)}{K}$ solving the equation
\begin{equation*}
	\tilde{\Phi}(|D\tilde{u}|,x)\tilde{F}(D^2 \tilde{u}, x)+\tilde{H}(D\tilde{u},x) ={f}(x) \quad  \text{in} \quad B_{1},
\end{equation*}
where 
\begin{equation*}
\begin{split}
	\tilde{F}(X,x):=&K^{-1}F\left(KX, x\right),\quad \tilde{H}(t,x):=
	H\left(Kt, x\right),\\
	\tilde{\Phi}(t,x):=&\frac{\tilde{\sigma}_{1}(t)+{a}(x)\tilde{\sigma}_{2}(t)}{t},\quad  
	\tilde{\sigma}_{i}(t):=\sigma_{i}\left(Kt\right),\;i=1,2.
\end{split}
\end{equation*}
Observe that $\tilde{F}$ is still a uniformly $(\lambda,\Lambda)$-elliptic operator and satisfies \eqref{A2}, and 
$$\tilde{\sigma}_{i}^{-1}(t):=K^{-1}\sigma_{i}^{-1}(t),\;\;i=1,2.$$
It follows from \eqref{A3}, \eqref{A5} and $K\geq 1$ that
\begin{equation*}
	\int_{0}^{1}\frac{\tilde{\sigma}_{i}^{-1}(t)}{t}\text{d}t\leq \int_{0}^{1}\frac{\sigma_{i}^{-1}(t)}{t}\text{d}t<\infty, \quad i=1,2,
\end{equation*}
\begin{equation*}
	\Abs{\tilde{H}(t,x)}\leq \mathcal{M}\left(1+\sigma_{1}\left(K|t|\right)\right)= {\mathcal{M}}\left(1+\tilde{\sigma}_{1}(|t|)\right). 
\end{equation*}
Therefore, $\tilde{u}$ solves an equation possessing the same structure as \eqref{4doublephasezhufangcheng} with $\Phi$ of the form \eqref{xintiaojian} and $\|\tilde{u}\|_{L^{\infty}\left(B_{1}\right)}\leq 1$.

By following the approach of the proof of \cite[Proposition 2.1]{Baasandorj2023AML} (see also \cite[Proposition 1.2]{B-Demengel2016}), we can see that $u$ is a viscosity solution of 
\begin{equation}\label{zhuanhua56563}
	\bigg(\sigma_{1}(|Du|)+a(x)\sigma_{2}(|Du|)\bigg) F(D^2 u, x)=\hat{f}(x) \quad  \text{in} \quad B_{1/2},
\end{equation}
where $
 \hat{f}(x):=|Du|{f}(x)-|Du|{H}(Du,x)$. Moreover, we obtain from Proposition \ref{qiyituihualipjinxing} that $u$ is a Lipschitz continuous function, and as a consequence, the gradient $Du$ is bounded almost everywhere. Thus, we can estimate
\begin{equation*}
	\|\hat{f}\|_{L^{\infty}(B_{1/2})}\leq C \|{f}\|_{L^{\infty}(B_{1/2})}+C\mathcal{M}\bigg(1+\sigma_{1}(C)\bigg).
\end{equation*}
Therefore, we are able to apply Theorem \ref{main} together
with standard covering arguments to get the desired result. 
\end{proof}
Finally, we turn our attention to the proof of Theorem \ref{main333} regarding $C^{1,\alpha}$ regularity.
\begin{proof}[Proof of Theorem \ref{main333}]
As argued before, we can assume $\|u\|_{L^{\infty}(B_{1})}\leq 1$ and see that $u$ is a viscosity solution of the equation
\begin{equation*}
 F(D^2 u, x)=\hat{f}(x) \quad  \text{in} \quad B_{1/2},
\end{equation*}
where $\hat{f}(x):=|Du|\frac{{f}(x)-H(Du,x)}{\sigma_{1}(|Du|)+a(x)\sigma_{2}(|Du|)}$. It follows from \eqref{4aaquyuwuqiongtiaojian} that there exists constant $t_{0}>0$ such that 
\begin{equation}\label{44final}
	\frac{\sigma_{1}(t)}{t},\frac{\sigma_{2}(t)}{t}\geq \frac{\kappa_{0}}{2} \quad {\rm for}\;\; 0<t<t_{0}.
\end{equation}
Applying Proposition \ref{qiyituihualipjinxing}, in combination with \eqref{44final}, \eqref{A3}-\eqref{A5}, we arrive at
$$\|\hat{f}\|_{L^{\infty}({B_{1/2}})}
\leq \max\left\{\frac{2}{\kappa_{0}},\frac{C}{\sigma_{1}(t_{0})}\right\}\left(\|f\|_{L^{\infty}({B_{1/2}})}+\mathcal{M}(1+\sigma_{1}(C))\right).$$
Consequently, by following the standard covering arguments and the classical regularity results (cf. \cite{Caff1,Caffarelli1989}), we conclude that $u\in C^{1,\alpha}_{\rm loc}(B_{1})$ for some $\alpha\in (0,1)$.
\end{proof}
\section*{Acknowledgments}
This work is supported by the National Natural Science Foundation of China (NSFC Grant No.12571103), Young Scientific and Technological Talents (Level Three) in Tianjin and Tianjin Natural Science Foundation Project (No. 25JCQNJC01400).
\section*{Data availability} Data sharing is not applicable to this article as obviously no datasets were generated or analyzed during the current study.
\section*{Conflict of interest} Author states no conflict of interest.

\end{document}